\documentclass[reqno]{amsart}

\usepackage{amssymb,amsfonts,amstext,amsthm,hyperref,cleveref,xcolor}

\newtheorem{thrm}{Theorem}[section]
\newtheorem{cor}[thrm]{Corollary}
\newtheorem{lem}[thrm]{Lemma}
\newtheorem{prop}[thrm]{Proposition}

\theoremstyle{definition}
\newtheorem{defn}[thrm]{Definition}

\newtheorem{rem}[thrm]{Remark}

\crefrangeformat{equation}{#3(#1)#4--#5(#2)#6}

\crefname{thrm}{Theorem}{Theorems}
\crefname{lem}{Lemma}{Lemmas}
\crefname{cor}{Corollary}{Corollaries}
\crefname{prop}{Proposition}{Propositions}
\crefname{defn}{Definition}{Definitions}
\crefname{exm}{Example}{Examples}
\crefname{rem}{Remark}{Remarks}
\crefname{section}{Section}{Sections}
\crefname{equation}{\unskip}{\unskip}
\crefname{enumi}{\unskip}{\unskip}

\renewcommand{\iff}{\Leftrightarrow}

\newcommand{\vf}{\varphi}
\newcommand{\af}{\alpha}
\newcommand{\bt}{\beta}

\newcommand{\0}{\theta}
\newcommand{\e}{\epsilon}

\newcommand{\tl}{\tilde}
\newcommand{\x}{\bar{x}}

\newcommand{\cC}{\mathcal C}
\newcommand{\ob}[1]{\operatorname{\mathrm{Ob}}{#1}}
\newcommand{\Mor}[1]{\operatorname{\mathrm{Mor}}(#1)}
\newcommand{\mor}[2]{\operatorname{\mathrm{Mor}}(#1,#2)}
\newcommand{\End}[1]{\operatorname{\mathrm{End}}(#1)}
\newcommand{\id}{\mathrm{id}}
\newcommand{\sst}{\subseteq}

\begin{document}

\title[Jordan Isomorphisms of the Finitary Incidence Ring of a Pocategory]{Jordan Isomorphisms\\ of the Finitary Incidence Ring\\ of a Partially Ordered Category}

\author{Rosali Brusamarello}
\address{Departamento de Matem\'atica, Universidade Estadual de Maring\'a, Maring\'a --- PR, CEP: 87020--900, Brazil}
\email{brusama@uem.br}

\author{\'Erica Z. Fornaroli}
\address{Departamento de Matem\'atica, Universidade Estadual de Maring\'a, Maring\'a --- PR, CEP: 87020--900, Brazil}
\email{ezancanella@uem.br}

\author{Mykola Khrypchenko}
\address{Departamento de Matem\'atica, Universidade Federal de Santa Catarina,  Campus Reitor Jo\~ao David Ferreira Lima, Florian\'opolis --- SC, CEP: 88040--900, Brazil}
\email{nskhripchenko@gmail.com}

\begin{abstract}
	Let $\cC$ be a pocategory, $FI(\cC)$ the finitary incidence ring of $\cC$ and $\vf$ a Jordan isomorphism of $FI(\cC)$ onto an associative ring $A$. We study the problem of decomposition of $\vf$ into the (near-)sum of a homomorphism and an anti-homomorphism. In particular, we obtain generalizations of the main results of \cite{Akkurts-Barker,BFK}.
\end{abstract}



\maketitle

\section*{Introduction}\label{intro}

The study of Jordan maps on the ring  of upper triangular matrices  was iniciated by L.~Moln\'ar and P.~\v{S}emrl~\cite{Molnar-Semrl98}. They proved in~\cite[Corollary 4]{Molnar-Semrl98} that each Jordan automorphism on the ring $T_n(R)$ of upper triangular matrices is either an automorphism or an anti-automorphism, where $ R$ is a field with at least $3$ elements. K.~I.~Beidar, M.~Bre\v{s}ar and M.~A.~Che\-bo\-tar generalized this result in ~\cite{BeBreChe}, where they considered the case when $R$ is a $2$-torsionfree unital commutative ring without non-trivial idempotents and $n\ge 2$ and showed that each Jordan isomorphism of $T_n( R)$ onto a $R$-algebra is either an isomorphism or an anti-isomorphism. D.~Benkovi\v{c}, using the new notion of near-sum, described in ~\cite[Theorem 4.1]{Benkovic05} all Jordan homomorphisms $T_n(R)\to A$, where $R$ is an arbitrary $2$-torsionfree commutative ring and $A$ is an $R$-algebra. E.~Akkurt, M.~Akkurt and G.~P.~Barker extended in~\cite[Theorem 2.1]{Akkurts-Barker} Benkovi\v{c}'s result to structural matrix algebras $T_n(R,\rho)$, where $\rho$ is either a partial order or a quasi-order each of whose equivalence classes contains at least $2$ elements.

It should be noted that the structural matrix algebra $T_n( R,\rho)$ is isomorphic to the incidence algebra $I(P,R)$ of the ordered set $P=(\{1, \ldots, n\},\rho)$ over $R$ as defined in ~\cite{SpDo}. A generalization of incidence algebras to the case of non-locally finite posets appeared in ~\cite{Khripchenko-Novikov09}, where the authors defined the so-called finitary incidence algebras $FI(P,R)$. R.~Brusamarello, E.~Z.~Fornaroli and M.~Khrypchenko~\cite{BFK} extended one of the main results of~\cite[Theorem 2.1]{Akkurts-Barker} to the case of $FI(P,R)$, namely, they showed that each $R$-linear Jordan isomorphism of $FI(P,R)$ onto an $R$-algebra $A$ is the near-sum of a homomorphism and an anti-homomorphism, where $P$ is an arbitrary partially ordered set and $R$ is a commutative $2$-torsionfree unital ring. 

Our initial goal was to generalize~\cite[Theorem 3.13]{BFK} to the case of a quasiordered set $P$. This case is technically more complicated, and to deal with it, we used the notion of the finitary incidence ring $FI(\cC)$ of a pocategory, introduced by M.~Khrypchenko in~\cite{Khr10-quasi}. This permitted to us to obtain generalizations of both~\cite[Theorem 2.1]{Akkurts-Barker} and~\cite[Theorem 3.13]{BFK} at the same time. Moreover, the methods elaborated in our paper work for non-necessarily $R$-linear Jordan isomorphisms of incidence rings.

The structure of the article is as follows.

In \cref{sec-prelim} we recall the definitions and the main properties of Jordan homomorphisms and finitary incidence algebras. We recall, for instance, that as abelian group, $FI(\cC)=D(\cC)\oplus FZ(\cC)$, where $D(\cC)$ is the subring of $FI(\cC)$ of all diagonal elements and  $FZ(\cC)$ is the ideal of $FI(\cC)$ consisting of $\af\in FI(\cC)$ with $\af_{xy}=0_{xy}$ for $x=y$.

In \cref{sec-jiso-FI(C)} we study the decomposition of a Jordan isomorphism $\vf:FI(\cC) \to A$ into a near-sum. First we show that  $\vf|_{FZ(\cC)}$ decomposes as the sum of two additive maps $\psi, \0:FI(\cC)\to A$. Then, using several technical lemmas, we prove in \cref{psi-and-0-hom-and-anti-hom} that the maps $\psi$ and $\0$ are a homomorphism and an anti-homomorphism, respectively. The main result of \cref{sec-jiso-FI(C)} is \cref{vf-near-sum-of-tl-psi-and-tl-theta}, which says that  $\vf:FI(\cC) \to A$ the near-sum of two additive maps $\tl\psi,\tl\0:FI(\cC)\to A$ with respect to $D(\cC)$ and $FZ(\cC)$. Moreover, we give necessary and sufficient conditions under which $\tl\psi$ and $\tl\0$ are a homomorphism and an anti-homomorphism, respectively.

\cref{jord-iso-FI(P_R)} is devoted to the decomposition of a Jordan isomorphism $\vf:FI(\cC) \to A$ as the sum of a homomorphism and an anti-homomorphism. In \cref{vf|_D(C)-is-the-sum-of-psi-and-theta} we give necessary and sufficient conditions under which $\vf|_{D(C)}$ admits such a decomposition. Finally, we restrict ourselves to the case $\cC=\cC(P,R)$, where $P$ is a quasiordered set such that $1<|\bar{x}|<\infty$ for all $\bar{x}\in \bar{P}$, $R$ is a commutative ring and $A$ is an $R$-algebra. We prove in \cref{vf=psi+0} that each $R$-linear Jordan isomorphism $\vf:FI(P,R) \to A$ is the sum of a homomorphism and an anti-homomorphism.

\section{Preliminaries}\label{sec-prelim}

\subsection{Jordan homomorphisms}\label{jordan-homo}
Let $R$ and $S$ be associative rings. An additive map $\varphi:R\to S$ is called a \emph{Jordan homomorphism}, if it satisfies
\begin{align}
\vf(r^2)&=\vf(r)^2,\label{vf(r^2)}\\
\vf(rsr)&=\vf(r)\vf(s)\vf(r),\label{vf(rsr)}
\end{align}
for all $r,s\in R$. A bijective Jordan homomorphism is called a \emph{Jordan isomorphism}.

Each homomorphism, as well as an anti-homomorphism, is a Jordan homomorphism. The sum of a homomorphism $\psi:R\to S$ and an anti-homomorphism $\0:R\to
S$ is a Jordan homomorphism, if $\psi(r)\0(s)=\0(s)\psi(r)=0$ for all $r,s\in R$. A more general construction was introduced by D.~Benkovi\v{c} in
\cite{Benkovic05}. Suppose that $R$ can be represented as the direct sum of additive subgroups $R_0\oplus R_1$, where $R_0$ is a subring of $R$ and
$R_1$ is an ideal of $R$. Let $\psi,\0:R\to S$ be additive maps, such that $\psi|_{R_0}=\0|_{R_0}$ and $\psi(r)\0(s)=\0(s)\psi(r)=0$ for all $r,s\in
R_1$. Then the \emph{near-sum} of $\psi$ and $\0$ (with respect to $R_0$ and $R_1$) is the additive map $\vf:R\to S$, which satisfies
$\vf|_{R_0}=\psi|_{R_0}=\0|_{R_0}$ and  $\vf|_{R_1}=\psi|_{R_1}+\0|_{R_1}$. If $\psi$ is a homomorphism and $\0$ is an anti-homomorphism, then one can
show that $\vf$ is a Jordan homomorphism in this case.

Let us now mention some basic facts on Jordan homomorphisms. Applying $\vf$ to $(r+s)^2$ and using \cref{vf(r^2)}, we get
\begin{align}\label{vf(rs+sr)}
 \vf(rs+sr)=\vf(r)\vf(s)+\vf(s)\vf(r).
\end{align}
The substitution of $r$ by $r+t$ in \cref{vf(rsr)} gives
\begin{align*}%\label{vf(rst+tsr)}
 \vf(rst+tsr)=\vf(r)\vf(s)\vf(t)+\vf(t)\vf(s)\vf(r).
\end{align*}
We shall also use the following fact (see Corollary 2 of \cite[Theorem 1]{Jacobson-Rickart50}). If $e$ is an idempotent, such that $er=re$, then
\begin{align}\label{vf(e)vf(r)=vf(r)vf(e)}
 \vf(r)\vf(e)=\vf(e)\vf(r)=\vf(re).
\end{align}
In particular, if $R$ has identity $1$, then $\vf(1)$ is the identity of $\vf(R)$. Another particular case of \cref{vf(e)vf(r)=vf(r)vf(e)}: if
$er=re=0$, then
\begin{align}\label{vf(e)vf(r)=0}
 \vf(e)\vf(r)=\vf(r)\vf(e)=0.
\end{align}

From now on, all rings will be associative and with $1$.

\subsection{Finitary incidence rings}\label{fin-inc-ring}
Recall from \cite{Khr16} (see also \cite{Khr10-quasi} for a slightly stronger definition) that a \emph{pocategory} is a preadditive small category $\cC$
with a partial order $\le$ on the set $\ob\cC$ of its objects. Denote by $I(\cC)$ the set of the formal sums
\begin{align}\label{formal-sum-in-I(C)}
\alpha=\sum_{x\le y}\alpha_{xy}e_{xy},
\end{align}
where $x,y\in\ob\cC$, $\alpha_{xy}\in\mor xy$ and $e_{xy}$ is a symbol. It is an abelian group under the addition coming from the addition of morphisms
in $\cC$. We shall also consider the series $\alpha$ of the form~\cref{formal-sum-in-I(C)}, whose indices run through a subset $X$ of the ordered pairs
$(x,y)$, $x,y\in\ob\cC$, $x\le y$, in which case $\alpha_{xy}$ will be meant to be the zero $0_{xy}$ of $\mor xy$ for $(x,y)\not\in X$.

The sum~\cref{formal-sum-in-I(C)} is called a {\it finitary series}, whenever for any pair of $x,y\in\ob\cC$ with $x<y$ there exists only a finite
number of $u,v\in\ob\cC$, such that $x\le u<v\le y$ and $\alpha_{uv}\ne 0_{uv}$. The set of finitary series, denoted by $FI(\cC)$, is an additive
subgroup of $I(\cC)$, and it is closed under the convolution
\begin{align*}%\label{conv-in-I(C)}
\alpha\beta=\sum_{x\le y}\left(\sum_{x\le z\le y}\alpha_{xz}\beta_{zy}\right)e_{xy},
\end{align*}
where $\alpha,\beta\in FI(\cC)$. Thus, $FI(\cC)$ is a ring, called the {\it finitary incidence ring of $\cC$}. The identity element $\delta$ of
$FI(\cC)$ is the series $\delta=\sum_{x\in\ob{\cC}}\id_xe_{xx}$, where $\id_x$ is the identity morphism from $\End x$.

An element $\af\in FI(\cC)$ will be said to be {\it diagonal}, if $\af_{xy}=0_{xy}$ for $x\ne y$. The  subring of $FI(\cC)$ consisting of the diagonal
elements will be denoted by $D(\cC)$. Clearly, as an abelian group, $FI(\cC)=D(\cC)\oplus FZ(\cC)$, where $FZ(\cC)$ is the ideal of $FI(\cC)$ consisting
of $\af\in FI(\cC)$ with $\af_{xy}=0_{xy}$ for $x=y$. Thus, each $\af\in FI(\cC)$ can be uniquely decomposed as $\af=\af_D+\af_Z$, where $\af_D\in
D(\cC)$ and $\af_Z\in FZ(\cC)$.

Observe that
\begin{align}\label{af_xye_xy-cdot-bt_uve_uv}
\af_{xy}e_{xy}\cdot \bt_{uv}e_{uv}=
 \begin{cases}
  \af_{xy}\bt_{uv}e_{xv}, & \mbox{if $y=u$},\\
  0, & \mbox{otherwise}.
 \end{cases}
\end{align}
In particular, the elements $e_x:=\id_x e_{xx}$, $x\in \ob(\cC)$, are pairwise orthogonal idempotents of $FI(\cC)$, and for any $\af\in FI(\cC)$
\begin{align}\label{e_x-alpha-e_y}
 e_x\af e_y=\begin{cases}
          \af_{xy}e_{xy}, & \mbox{ if }x\le y,\\
          0, & \mbox{ otherwise}.
         \end{cases}
\end{align}
Consequently, for all $\af,\bt\in FI(\cC)$
\begin{align*}%\label{af=bt<=>e_x-af-e_y+e_y-af-e_x=e_x-bt-e_y+e_y-bt-e_x}
 \af=\bt&\iff \forall x\le y:\ e_x\af e_y=e_x\bt e_y\notag\\
 &\iff \begin{cases}
        \forall x<y:& e_x\af e_y+e_y\af e_x=e_x\bt e_y+e_y\bt e_x,\\
        \forall x:& e_x\af e_x=e_x\bt e_x.
       \end{cases}
\end{align*}

Given $X\subseteq \ob{\cC}$, we shall use the notation $e_X$ for the diagonal idempotent $\sum_{x\in X}\id_x e_{xx}$. In particular, $e_x=e_{\{x\}}$.
Note that $e_Xe_Y=e_{X\cap Y}$, so $e_xe_X=e_x$ for $x\in X$, and $e_xe_X=0$ otherwise.

Let $(P,\preceq)$ be a preordered set and $R$ a ring. We recall the construction of the pocategory $\cC(P,R)$, introduced in \cite{Khr16}.

Denote by $\sim$ the natural equivalence relation on $P$, namely, $x\sim y\Leftrightarrow x\preceq y\preceq x$, and by $\bar P$ the quotient set $P/\sim$. Define $\ob\cC(P,R)$ to be $\bar P$ with the induced partial order $\leq$. For any pair $\bar x,\bar y \in \ob\cC(P,R)$, let $\Mor {\bar{x},\bar{y}}=RFM_{\bar x\times \bar y}(R)$, where $RFM_{\bar I\times \bar J}(R)$ denotes the additive group of row-finite matrices over $R$, whose rows are indexed by the elements of $I$ and columns by the elements of $J$. The composition of morphisms in $\cC(P,R)$ is the matrix multiplication, which is defined by the row-finiteness condition.

The \emph{finitary incidence ring of $P$ over $R$}, denoted by $FI(P,R)$, is by definition $FI(\cC(P,R))$. Furthermore, $D(P,R):=D(\cC(P,R))$ and $FZ(P,R):=FZ(\cC(P,R))$.

\section{Jordan isomorphisms of $FI(\cC)$}\label{sec-jiso-FI(C)}

Let $A$ be a ring, $\cC$ an arbitrary pocategory and $\vf:FI(\cC)\to A$ a Jordan isomorphism. We first adapt the ideas from \cite{Akkurts-Barker,BFK} to decompose $\vf|_{FZ(\cC)}$ as the sum of two additive maps $\psi,\0:FZ(\cC)\to A$.

\subsection{Decomposition of $\vf|_{FZ(\cC)}$}
The following two lemmas are straightforward generalizations of \cite[Lemmas 3.1--3.2]{BFK}, so we omit their proofs.
\begin{lem}
	Let $\vf:FI(\cC)\to A$ be a Jordan homomorphism. Then for any $\af\in FI(\cC)$ one has
	\begin{align}
	\forall x<y:\
\vf(\af_{xy}e_{xy})&=\vf(e_x)\vf(\af)\vf(e_y)+\vf(e_y)\vf(\af)\vf(e_x),\label{vf(af_xye_xy)=vf(e_x)vf(af)vf(e_y)+vf(e_y)vf(af)vf(e_x)}\\
	\forall x:\ \vf(\af_{xx}e_{xx})&=\vf(e_x)\vf(\af)\vf(e_x).\label{vf(af_xxe_x)=vf(e_x)vf(af)vf(e_x)}
	\end{align}
\end{lem}

\begin{lem}\label{a=b-in-A}
	Let $\vf:FI(\cC)\to A$ be a Jordan isomorphism. Then for all $a,b\in A$
	$$
	a=b\iff\begin{cases}
	\forall x<y:& \vf(e_x)a\vf(e_y)+\vf(e_y)a\vf(e_x)=\vf(e_x)b\vf(e_y)+\vf(e_y)b\vf(e_x),\\
	\forall x:& \vf(e_x)a\vf(e_x)=\vf(e_x)b\vf(e_x).
	\end{cases}
	$$
\end{lem}

Observe also from \cref{vf(af_xye_xy)=vf(e_x)vf(af)vf(e_y)+vf(e_y)vf(af)vf(e_x),vf(e)vf(r)=0} that
\begin{align}
	\vf(e_x)\vf(\af)\vf(e_y)&=\vf(e_x)\vf(\af_{xy}e_{xy})\vf(e_y),\label{vf(e_x)vf(af)vf(e_y)=vf(e_x)vf(af_xye_xy)vf(e_y)}\\
	\vf(e_y)\vf(\af)\vf(e_x)&=\vf(e_y)\vf(\af_{xy}e_{xy})\vf(e_x),\label{vf(e_y)vf(af)vf(e_x)=vf(e_y)vf(af_xye_xy)vf(e_x)}
\end{align}
for all $x<y$.

The next two lemmas will lead us to the definition of $\psi$ and $\0$.
\begin{lem}\label{terms-of-psi-and-0}
	Given a Jordan isomorphism $\vf:FI(\cC)\to A$, $\af\in FI(\cC)$ and $x<y$, there exists a (unique) pair of $\af'_{xy},\af''_{xy}\in\mor xy$, such that
	\begin{align}
		\vf(e_x)\vf(\af)\vf(e_y)=\vf(\af'_{xy}e_{xy}),\label{defn-of-af'_xy}\\
		\vf(e_y)\vf(\af)\vf(e_x)=\vf(\af''_{xy}e_{xy}).\label{defn-of-af''_xy}		
	\end{align}
\end{lem}
\begin{proof}
	We construct $\af'_{xy}$, the construction of $\af''_{xy}$ is similar. Since $\vf$ is bijective, there exists a unique $\bt\in FI(\cC)$, such that
$\vf(e_x)\vf(\af)\vf(e_y)=\vf(\bt)$. Apply \cref{vf(af_xye_xy)=vf(e_x)vf(af)vf(e_y)+vf(e_y)vf(af)vf(e_x)} to $\bt$ and use \cref{vf(e)vf(r)=0,vf(r^2)}
to obtain
	\begin{align*}
		\vf(\bt_{xy}e_{xy})&=\vf(e_x)\vf(\bt)\vf(e_y)+\vf(e_y)\vf(\bt)\vf(e_x)\\
		&=\vf(e_x)\vf(e_x)\vf(\af)\vf(e_y)\vf(e_y)+\vf(e_y)\vf(e_x)\vf(\af)\vf(e_y)\vf(e_x)\\
		&=\vf(e_x)\vf(\af)\vf(e_y)\\
		&=\vf(\bt).
	\end{align*}
Now take $\af'_{xy}=\bt_{xy}$.
\end{proof}

\begin{lem}\label{af'-and-af''-are-finitary-series}
	Let $\vf:FI(\cC)\to A$ be a Jordan isomorphism, $\af\in FZ(\cC)$ and $\{\af'_{xy}\}_{x<y},\{\af''_{xy}\}_{x<y}\subseteq\Mor{\cC}$ defined by
\cref{defn-of-af'_xy,defn-of-af''_xy}. Then
	\begin{align}
		\af'&=\sum_{x<y}\af'_{xy}e_{xy},\label{defn-of-af'}\\
		\af''&=\sum_{x<y}\af''_{xy}e_{xy}\label{defn-of-af''}
	\end{align}
	are finitary series, such that
\begin{align}\label{af=af'+af''}
	\af=\af'+\af''.
\end{align}	
 Moreover, the maps $\af\mapsto\af'$ and $\af\mapsto\af''$ are additive.
\end{lem}
\begin{proof}
	We first prove \cref{af=af'+af''} and thus reduce all the assertions about $\af''$ to the corresponding assertions about $\af'$. Indeed,
\cref{af=af'+af''} easily follows from \cref{vf(af_xye_xy)=vf(e_x)vf(af)vf(e_y)+vf(e_y)vf(af)vf(e_x),defn-of-af'_xy,defn-of-af''_xy} and bijectivity of
$\vf$, since
	\begin{align*}
		\vf(\af_{xy}e_{xy})&=\vf(e_x)\vf(\af)\vf(e_y)+\vf(e_y)\vf(\af)\vf(e_x)\\
		&=\vf(\af'_{xy}e_{xy})+\vf(\af''_{xy}e_{xy})\\
		&=\vf((\af'_{xy}+\af''_{xy})e_{xy}).
	\end{align*}
	
	Now suppose that $\af'_{uv}\ne 0_{uv}$ for an infinite number of ordered pairs $x\le u<v\le y$. Then $\vf(e_u)\vf(\af)\vf(e_v)\ne 0$ by
\cref{defn-of-af'_xy}. It follows from \cref{vf(e_x)vf(af)vf(e_y)=vf(e_x)vf(af_xye_xy)vf(e_y)} that $\vf(\af_{uv}e_{uv})\ne 0$. Hence, $\af_{uv}\ne
0_{uv}$ for $x\le u<v\le y$ contradicting the fact that $\af\in FI(\cC)$. Thus, $\af'\in FI(\cC)$, and additivity of $\af\mapsto\af'$ is explained by
additivity of $\vf$ and distributivity of multiplication in $A$.
\end{proof}

Thus, with any Jordan isomorphism $\vf:FI(\cC)\to A$ we may associate $\psi,\0:FZ(\cC)\to A$ given by
\begin{align}
	\psi(\af)&=\vf(\af'),\label{defn-of-psi}\\
	\0(\af)&=\vf(\af''),\label{defn-of-0}
\end{align}
where $\af\in FZ(\cC)$ and $\af',\af''$ are defined by means of \cref{defn-of-af'_xy,defn-of-af''_xy,defn-of-af',defn-of-af''}. By
\cref{af'-and-af''-are-finitary-series} the maps $\psi$ and $\0$ are well defined and additive.

\begin{prop}\label{vf|_FZ-is-sum}
	Let $\vf:FI(\cC)\to A$ be a Jordan isomorphism. Then
	\begin{align*}
		\vf|_{FZ(\cC)}=\psi+\0.
	\end{align*}
\end{prop}
\begin{proof}
If $\af\in FZ(\cC)$, then $\psi(\af)=\vf(\af')$ and $\0(\af)=\vf(\af'')$ by \cref{defn-of-psi,defn-of-0}, so $\vf(\af)=\psi(\af)+\0(\af)$ thanks to
\cref{af=af'+af''}.
\end{proof}

\subsection{Properties of $\psi$ and $\0$}

In this subsection we prove that the maps $\psi$ and $\0$ are in fact a homomorphism and an anti-homomorphism.
We first show that they satisfy the properties analogous to the ones given in \cite[Propositions 3.5, 3.12]{BFK}, as the next lemma shows.

\begin{lem}\label{vf(e_x)psi(af)vf(e_y)-etc}
	Let $\vf:FI(\cC)\to A$ be a Jordan isomorphism, $\psi,\0:FZ(\cC)\to A$ the associated maps given by \cref{defn-of-psi,defn-of-0} and $\af\in
FZ(\cC)$. Then for all $x<y$
	\begin{align}
		\vf(e_x)\psi(\af)\vf(e_y)&=\vf(e_x)\vf(\af)\vf(e_y),\label{vf(e_x)psi(af)vf(e_y)=vf(e_x)vf(af)vf(e_y)}\\
		\vf(e_y)\0(\af)\vf(e_x)&=\vf(e_y)\vf(\af)\vf(e_x)\notag%\label{vf(e_y)0(af)vf(e_x)=vf(e_y)vf(af)vf(e_x)}
	\end{align}
	and
	\begin{align}\label{vf(e_y)psi(af)vf(e_x)=vf(e_x)0(af)vf(e_y)=vf(e_x)psi(af)vf(e_x)=vf(e_x)0(af)vf(e_x)=0}
		\vf(e_y)\psi(\af)\vf(e_x)=\vf(e_x)\psi(\af)\vf(e_x)=\vf(e_x)\0(\af)\vf(e_y)=\vf(e_x)\0(\af)\vf(e_x)=0.
	\end{align}
\end{lem}
\begin{proof}
    By \cref{vf(r^2),vf(e_x)vf(af)vf(e_y)=vf(e_x)vf(af_xye_xy)vf(e_y),defn-of-af'_xy}
	\begin{align*}
     \vf(e_x)\psi(\af)\vf(e_y)&=\vf(e_x)\vf(\af')\vf(e_y)=\vf(e_x)\vf(\af'_{xy}e_{xy})\vf(e_y)\\
     &=\vf(e_x)\vf(e_x)\vf(\af)\vf(e_y)\vf(e_y)=\vf(e_x)\vf(\af)\vf(e_y).
	\end{align*}
	Similarly, by \cref{vf(e)vf(r)=0,vf(e_y)vf(af)vf(e_x)=vf(e_y)vf(af_xye_xy)vf(e_x),defn-of-af'_xy}
	\begin{align*}
	\vf(e_y)\psi(\af)\vf(e_x)&=\vf(e_y)\vf(\af')\vf(e_x)=\vf(e_y)\vf(\af'_{xy}e_{xy})\vf(e_x)\\
	&=\vf(e_y)\vf(e_x)\vf(\af)\vf(e_y)\vf(e_x)=0.
	\end{align*}	
	Finally, by \cref{vf(af_xxe_x)=vf(e_x)vf(af)vf(e_x)}
	\begin{align*}
		\vf(e_x)\psi(\af)\vf(e_x)=\vf(e_x)\vf(\af')\vf(e_x)=\vf(\af'_{xx}e_{xx})=0,
	\end{align*}
	the latter equality being explained by the fact that $\af'\in FZ(\cC)$.
	
	The identities involving $\0$ are proved in an analogous way using
\cref{vf(af_xxe_x)=vf(e_x)vf(af)vf(e_x),vf(e_y)vf(af)vf(e_x)=vf(e_y)vf(af_xye_xy)vf(e_x),vf(e_x)vf(af)vf(e_y)=vf(e_x)vf(af_xye_xy)vf(e_y),defn-of-af''_xy}.
\end{proof}

The following lemma completes the previous one.

\begin{lem}\label{psi(af_x-e_xy)}
	Under the conditions of \cref{vf(e_x)psi(af)vf(e_y)-etc} one has
	\begin{align}
	\psi(\af_{xy}e_{xy})&=\vf(e_x)\psi(\af)\vf(e_y),\label{psi(af_xe-e_xy)=vf(e_x)psi(af)vf(e_y)}\\
	\0(\af_{xy}e_{xy})&=\vf(e_y)\0(\af)\vf(e_x)\label{0(af_xe-e_xy)=vf(e_y)0(af)vf(e_x)}
	\end{align}
	for all $x<y$.
\end{lem}
\begin{proof}
	Given arbitrary $u<v$ such that $(u,v)\ne(x,y)$, we see by
\cref{vf(e_x)psi(af)vf(e_y)=vf(e_x)vf(af)vf(e_y),vf(e_x)vf(af)vf(e_y)=vf(e_x)vf(af_xye_xy)vf(e_y)} that
	\begin{align*}
	\vf(e_u)\psi(\af_{xy}e_{xy})\vf(e_v)&= \vf(e_u)\vf(\af_{xy}e_{xy})\vf(e_v)=\vf(e_u)\vf((\af_{xy}e_{xy})_{uv}e_{uv})\vf(e_v),
	\end{align*}
	and $\vf(e_v)\psi(\af_{xy}e_{xy})\vf(e_u)=0$ thanks to \cref{vf(e_y)psi(af)vf(e_x)=vf(e_x)0(af)vf(e_y)=vf(e_x)psi(af)vf(e_x)=vf(e_x)0(af)vf(e_x)=0}.
Evidently,
	\begin{align*}
	\vf(e_u)\vf(e_x)\psi(\af)\vf(e_y)\vf(e_v)= 0 = \vf(e_v)\vf(e_x)\psi(\af)\vf(e_y)\vf(e_u)
	\end{align*}
	in view of \cref{vf(e)vf(r)=0}. Moreover,
	\begin{align*}
	\vf(e_x)\psi(\af_{xy}e_{xy})\vf(e_y)&=\vf(e_x)\vf(\af_{xy}e_{xy})\vf(e_y)=\vf(e_x)\vf(\af)\vf(e_y)\\
	&=\vf(e_x)\psi(\af)\vf(e_y)=\vf(e_x)^2\psi(\af)\vf(e_y)^2
	\end{align*}
	by \cref{vf(e_x)psi(af)vf(e_y)=vf(e_x)vf(af)vf(e_y),vf(e_x)vf(af)vf(e_y)=vf(e_x)vf(af_xye_xy)vf(e_y),vf(r^2)}, and
	\begin{align*}
	\vf(e_y)\psi(\af_{xy}e_{xy})\vf(e_x)= 0 = \vf(e_y)\vf(e_x)\psi(\af)\vf(e_y)\vf(e_x)
	\end{align*}
	by \cref{vf(e_y)psi(af)vf(e_x)=vf(e_x)0(af)vf(e_y)=vf(e_x)psi(af)vf(e_x)=vf(e_x)0(af)vf(e_x)=0,vf(e)vf(r)=0}. Finally,
	\begin{align*}
	\vf(e_u)\psi(\af_{xy}e_{xy})\vf(e_u)=0=\vf(e_u)\vf(e_x)\psi(\af)\vf(e_y)\vf(e_u)
	\end{align*}
	because of \cref{vf(e)vf(r)=0,vf(e_y)psi(af)vf(e_x)=vf(e_x)0(af)vf(e_y)=vf(e_x)psi(af)vf(e_x)=vf(e_x)0(af)vf(e_x)=0}. Thus,
\cref{psi(af_xe-e_xy)=vf(e_x)psi(af)vf(e_y)} holds by \cref{a=b-in-A}. The proof of \cref{0(af_xe-e_xy)=vf(e_y)0(af)vf(e_x)} is analogous.
\end{proof}

We shall also need a technical result, which deals with some kind of a restriction of a finitary series to a subset of ordered pairs.

\begin{defn}\label{defn-of-af|_X^Y}
	Given $\af\in FI(\cC)$ and $X,Y\subseteq\ob{\cC}$, define
	\begin{align}\label{af|_X^Y=sum_x-in-X-y-in-Y-af_xye_xy}
	\af|_X^Y=\sum_{x\in X, y\in Y, x\le y}\af_{xy}e_{xy}.
	\end{align}
\end{defn}
We shall write $\af|_x^Y$ for $\af|_{\{x\}}^Y$, $\af|_X^y$ for $\af|_X^{\{y\}}$ and $\af|_x^y$ for $\af|_{\{x\}}^{\{y\}}$. We shall also use the
following shorter notations:
\begin{align*}
	\af|_X:=\af|_X^{\ob{\cC}},\ \af|^X:=\af|_{\ob{\cC}}^X.
\end{align*}

Let us first consider some basic properties of this operation on finitary series.
\begin{lem}\label{properties-of-af|_X^Y}
	For all $\af,\bt\in FI(\cC)$ and $X,Y,U,V\subseteq\ob{\cC}$ one has
	\begin{enumerate}
		\item $(\af|_X^Y)|_U^V=\af|_{X\cap U}^{Y\cap V}$;\label{(af|_X^Y)_U^V}
		\item $(\af+\bt)|_X^Y=\af|_X^Y+\bt|_X^Y$;\label{(af+bt)|_X^Y=af|_X^Y+bt|_X^Y}
		\item $(\af\bt)|_X^Y=\af|_X\cdot\bt|^Y$.\label{(af-bt)|_X^Y=af|_X-bt|^Y}
	\end{enumerate}
\end{lem}
\begin{proof}
	Items \cref{(af|_X^Y)_U^V,(af+bt)|_X^Y=af|_X^Y+bt|_X^Y} are obvious. For \cref{(af-bt)|_X^Y=af|_X-bt|^Y} consider first a pair $x\le y$ with $x\in
X$ and $y\in Y$. Then
	\begin{align*}
		((\af\bt)|_X^Y)_{xy}=(\af\bt)_{xy}=\sum_{x\le z\le y}\af_{xz}\bt_{zy}=\sum_{x\le z\le
y}\left(\af|_X\right)_{xz}\left(\bt|^Y\right)_{zy}=\left(\af|_X\cdot\bt|^Y\right)_{xy}.
	\end{align*}
	Now let $u\le v$, such that $u\not\in X$. Then
	\begin{align}\label{((af-bt)|_X^Y)_uv=(af|_X-bt|^Y)_uv=0}
		((\af\bt)|_X^Y)_{uv}= 0 = (\af|_X\cdot\bt|^Y)_{uv},
	\end{align}
as $(\af|_X)_{uw}=0_{uw}$ for all $w\in\ob{\cC}$. Similarly \cref{((af-bt)|_X^Y)_uv=(af|_X-bt|^Y)_uv=0} holds, when $v\not\in Y$.
\end{proof}

\begin{lem}\label{vf(e_X)vf(af)vf(e_Y)-in-terms-of-af|_X^Y}
	Let $\vf:FI(\cC)\to A$ be a Jordan isomorphism and $\psi,\0:FZ(\cC)\to A$ the associated maps given by \cref{defn-of-psi,defn-of-0}. Then for all
$\af\in FZ(\cC)$ and $X,Y\subseteq\ob{\cC}$ one has
	\begin{align}
	\vf(e_X)\psi(\af)\vf(e_Y)&=\vf(e_X)\psi\left(\af|_X^Y\right)\vf(e_Y),\label{vf(e_X)psi(af)vf(e_Y)=vf(e_X)psi(af|_X^Y)vf(e_Y)}\\
	\vf(e_X)\0(\af)\vf(e_Y)&=\vf(e_X)\0\left(\af|_Y^X\right)\vf(e_Y).\label{vf(e_X)0(af)vf(e_Y)=vf(e_X)0(af|_Y^X)vf(e_Y)}
	\end{align}
\end{lem}
\begin{proof}
	Observe that $e_X=e_{X\setminus Y}+e_{X\cap Y}$, $e_Y=e_{Y\setminus X}+e_{X\cap Y}$ and
	\begin{align*}
	\af|_X^Y=\af|_{X\setminus Y}^{Y\setminus X}+\af|_{X\setminus Y}^{X\cap Y}+\af|_{X\cap Y}^{Y\setminus X}+\af|_{X\cap Y}^{X\cap Y},
	\end{align*}
	so in view of distributivity of multiplication it suffices to prove
\cref{vf(e_X)psi(af)vf(e_Y)=vf(e_X)psi(af|_X^Y)vf(e_Y),vf(e_X)0(af)vf(e_Y)=vf(e_X)0(af|_Y^X)vf(e_Y)} in the following two cases:
	\begin{enumerate}
		\item $X\cap Y=\emptyset$;\label{case-X-and-Y-disjoint}
		\item $X=Y$.\label{case-X=Y}
	\end{enumerate}
	
	\emph{Case \cref{case-X-and-Y-disjoint}.} Assume that $X$ and $Y$ are disjoint. To show that \cref{vf(e_X)psi(af)vf(e_Y)=vf(e_X)psi(af|_X^Y)vf(e_Y)}
holds, we apply \cref{a=b-in-A}. Notice that multiplying any side of \cref{vf(e_X)psi(af)vf(e_Y)=vf(e_X)psi(af|_X^Y)vf(e_Y)} by $\vf(e_u)$ on the left
and on the right, we get zero, as $e_u$ is orthogonal either to $e_X$ or to $e_Y$, or to both of them. Now take $u<v$ and consider
	\begin{align}\label{vf(e_u)vf(e_X)vf(af)vf(e_Y)vf(e_v)+vf(e_v)vf(e_X)vf(af)vf(e_Y)vf(e_u)}
	\vf(e_u)\vf(e_X)\psi(\af)\vf(e_Y)\vf(e_v)+\vf(e_v)\vf(e_X)\psi(\af)\vf(e_Y)\vf(e_u).
	\end{align}
	If $u\not\in X\sqcup Y$ or $v\not\in X\sqcup Y$, then \cref{vf(e_u)vf(e_X)vf(af)vf(e_Y)vf(e_v)+vf(e_v)vf(e_X)vf(af)vf(e_Y)vf(e_u)} is zero, since
$e_u$ and $e_v$ are orthogonal both to $e_X$ and to $e_Y$. If $u,v\in X$, then $u,v\not\in Y$, i.e. $e_u$ and $e_v$ are orthogonal to $e_Y$, so
\cref{vf(e_u)vf(e_X)vf(af)vf(e_Y)vf(e_v)+vf(e_v)vf(e_X)vf(af)vf(e_Y)vf(e_u)} is again zero. By symmetry the same holds, when $u,v\in Y$. If $u\in Y$ and
$v\in X$, then $u\not\in X$ and $v\not\in Y$, hence by \cref{vf(e)vf(r)=vf(r)vf(e)} equality
\cref{vf(e_u)vf(e_X)vf(af)vf(e_Y)vf(e_v)+vf(e_v)vf(e_X)vf(af)vf(e_Y)vf(e_u)} becomes $\vf(e_v)\psi(\af)\vf(e_u)$, which is zero thanks to
\cref{vf(e_y)psi(af)vf(e_x)=vf(e_x)0(af)vf(e_y)=vf(e_x)psi(af)vf(e_x)=vf(e_x)0(af)vf(e_x)=0}. Notice also that all these four subcases do not depend on
$\af$, so everything remains valid with $\af$ replaced by $\af|_X^Y$. Finally, let $u\in X$ and $v\in Y$. Then $u\not\in Y$, $v\not\in X$, and thus
\cref{vf(e_u)vf(e_X)vf(af)vf(e_Y)vf(e_v)+vf(e_v)vf(e_X)vf(af)vf(e_Y)vf(e_u)} equals $\vf(e_u)\psi(\af)\vf(e_v)$. The latter is
	\begin{align}
		\vf(e_u)\vf(\af)\vf(e_v)&=\vf(e_u)\vf(\af_{uv}e_{uv})\vf(e_v)=\vf(e_u)\vf((\af|_X^Y)_{uv}e_{uv})\vf(e_v)\notag\\
		&=\vf(e_u)\vf(\af|_X^Y)\vf(e_v)=\vf(e_u)\psi(\af|_X^Y)\vf(e_v)\label{vf(e_u)vf(af)vf(e_v)=vf(e_u)psi(af_X^Y)vf(e_v)}
	\end{align}
	according to \cref{vf(e_x)psi(af)vf(e_y)=vf(e_x)vf(af)vf(e_y),vf(e_x)vf(af)vf(e_y)=vf(e_x)vf(af_xye_xy)vf(e_y),af|_X^Y=sum_x-in-X-y-in-Y-af_xye_xy}.
If, maintaining the assumptions on $u$ and $v$, we substitute $\af|_X^Y$ for $\af$ in
\cref{vf(e_u)vf(e_X)vf(af)vf(e_Y)vf(e_v)+vf(e_v)vf(e_X)vf(af)vf(e_Y)vf(e_u)}, then we get $\vf(e_u)\psi((\af|_X^Y)|_X^Y)\vf(e_v)$. But this is the same
as $\vf(e_u)\psi(\af|_X^Y)\vf(e_v)$ by \cref{(af|_X^Y)_U^V} of \cref{properties-of-af|_X^Y}, which in view of
\cref{vf(e_u)vf(af)vf(e_v)=vf(e_u)psi(af_X^Y)vf(e_v)}  completes the proof of \cref{vf(e_X)psi(af)vf(e_Y)=vf(e_X)psi(af|_X^Y)vf(e_Y)}. The same
technique is used to prove \cref{vf(e_X)0(af)vf(e_Y)=vf(e_X)0(af|_Y^X)vf(e_Y)}, and in this situation the only non-trivial subcase will be $u\in Y$ and
$v\in X$, which explains why $X$ and $Y$ are ``switched'' in the right-hand side of \cref{vf(e_X)0(af)vf(e_Y)=vf(e_X)0(af|_Y^X)vf(e_Y)}.
	
	\emph{Case \cref{case-X=Y}.} Let $X=Y$. For \cref{vf(e_X)psi(af)vf(e_Y)=vf(e_X)psi(af|_X^Y)vf(e_Y)} we shall again use \cref{a=b-in-A} skipping some
details, as the structure of the proof will be similar to the one in Case \cref{case-X-and-Y-disjoint}. For any $u\in\ob{\cC}$ the multiplication of any
side of \cref{vf(e_X)psi(af)vf(e_Y)=vf(e_X)psi(af|_X^Y)vf(e_Y)} by $e_u$ on the left and on the right gives zero either by \cref{vf(e)vf(r)=0}, when
$u\not\in X$, or by \cref{vf(e_y)psi(af)vf(e_x)=vf(e_x)0(af)vf(e_y)=vf(e_x)psi(af)vf(e_x)=vf(e_x)0(af)vf(e_x)=0}, when $u\in X$. Now, given $u<v$, we
see that both of the summands of \cref{vf(e_u)vf(e_X)vf(af)vf(e_Y)vf(e_v)+vf(e_v)vf(e_X)vf(af)vf(e_Y)vf(e_u)} are zero for any $\af$, when
$\{u,v\}\not\subseteq X$. If $u,v\in X$, then \cref{vf(e_u)vf(e_X)vf(af)vf(e_Y)vf(e_v)+vf(e_v)vf(e_X)vf(af)vf(e_Y)vf(e_u)} reduces to
$\vf(e_u)\psi(\af)\vf(e_v)+\vf(e_v)\psi(\af)\vf(e_u)$, whose second summand is zero by
\cref{vf(e_y)psi(af)vf(e_x)=vf(e_x)0(af)vf(e_y)=vf(e_x)psi(af)vf(e_x)=vf(e_x)0(af)vf(e_x)=0}, and the first one is $\vf(e_u)\psi(\af|_X^Y)\vf(e_v)$ as
in \cref{vf(e_u)vf(af)vf(e_v)=vf(e_u)psi(af_X^Y)vf(e_v)}. Since $\vf(e_v)\psi(\af|_X^Y)\vf(e_u)$ is also zero by
\cref{vf(e_y)psi(af)vf(e_x)=vf(e_x)0(af)vf(e_y)=vf(e_x)psi(af)vf(e_x)=vf(e_x)0(af)vf(e_x)=0}, then
	\begin{align*}
		\vf(e_u)\psi(\af)\vf(e_v)+\vf(e_v)\psi(\af)\vf(e_u)=\vf(e_u)\psi(\af|_X^Y)\vf(e_v)+\vf(e_v)\psi(\af|_X^Y)\vf(e_u),
	\end{align*}
	and we are done as in Case \cref{case-X-and-Y-disjoint}. The proof of \cref{vf(e_X)0(af)vf(e_Y)=vf(e_X)0(af|_Y^X)vf(e_Y)} is totally symmetric to
the proof of \cref{vf(e_X)psi(af)vf(e_Y)=vf(e_X)psi(af|_X^Y)vf(e_Y)}.
\end{proof}

In particular, taking $X=\ob{\cC}$ or $Y=\ob{\cC}$ in \cref{vf(e_X)vf(af)vf(e_Y)-in-terms-of-af|_X^Y} and observing that $e_{\ob{\cC}}$ is the identity
element which is preserved by $\vf$, we get the following formulas.
\begin{cor}\label{vf(e_X)psi(af)-etc}
 Let $\vf:FI(\cC)\to A$ be a Jordan isomorphism and $\psi,\0:FZ(\cC)\to A$ the associated maps given by \cref{defn-of-psi,defn-of-0}. Then for all
 $\af\in FZ(\cC)$ and $X\subseteq\ob{\cC}$
\begin{align}
	\vf(e_X)\psi(\af)&=\vf(e_X)\psi\left(\af|_X\right),\label{vf(e_X)psi(af)=vf(e_X)psi(af|_X)}\\
	\psi(\af)\vf(e_X)&=\psi\left(\af|^X\right)\vf(e_X),\label{psi(af)vf(e_X)=psi(af|^X)vf(e_X)}\\
	\vf(e_X)\0(\af)&=\vf(e_X)\0\left(\af|^X\right),\label{vf(e_X)0(af)=vf(e_X)0(af|^X)}\\
	\0(\af)\vf(e_X)&=\0\left(\af|_X\right)\vf(e_X).\label{0(af)vf(e_X)=0(af|_X)vf(e_X)}
\end{align}
\end{cor}

The next lemma shows that the maps $\af\mapsto\af'$ and $\af\mapsto\af''$ are compatible  with the operation defined by
\cref{af|_X^Y=sum_x-in-X-y-in-Y-af_xye_xy}.

\begin{lem}\label{'-of-restriction}
	For any $\af\in FZ(\cC)$ and $U,V\sst\ob{\cC}$ one has
	\begin{align}
		\left(\af|_U^V\right)'&=\af'|_U^V,\label{(af|_U^V)'=af'|_U^V}\\
		\left(\af|_U^V\right)''&=\af''|_U^V.\label{(af|_U^V)''=af''|_U^V}
	\end{align}
\end{lem}
\begin{proof}
	We first observe that for all $x<y$
	\begin{align*}
		\vf((\af_{xy}e_{xy})')=\psi(\af_{xy}e_{xy})=\vf(e_x)\psi(\af)\vf(e_y)=\vf(e_x)\vf(\af)\vf(e_y)=\vf(\af'_{xy}e_{xy})
	\end{align*}
	by \cref{psi(af_xe-e_xy)=vf(e_x)psi(af)vf(e_y),vf(e_x)psi(af)vf(e_y)=vf(e_x)vf(af)vf(e_y),defn-of-af'_xy,defn-of-psi}. Since $\vf$ is bijective, it
follows that
	\begin{align}\label{(af_xye_xy)'=af'_xye_xy}
		(\af_{xy}e_{xy})'=\af'_{xy}e_{xy}.
	\end{align}
	Now let $u\in U$ and $v\in V$. Then using \cref{(af_xye_xy)'=af'_xye_xy} we have
	\begin{align*} \left(\af|_U^V\right)'_{uv}=\left(\left(\af|_U^V\right)_{uv}e_{uv}\right)'_{uv}=\left(\af_{uv}e_{uv}\right)'_{uv}=\af'_{uv}=\left(\af'|_U^V\right)_{uv}.
	\end{align*}
	And if $u\not\in U$ or $v\not\in V$, then
	\begin{align*}
	\left(\af|_U^V\right)'_{uv}=\left(\left(\af|_U^V\right)_{uv}e_{uv}\right)'_{uv}=0_{uv}=\left(\af'|_U^V\right)_{uv}.
	\end{align*}
	Here we applied \cref{(af_xye_xy)'=af'_xye_xy} and used that $\left(\af|_U^V\right)_{uv}=0_{uv}= \left(\af'|_U^V\right)_{uv}$ . This proves
\cref{(af|_U^V)'=af'|_U^V}. Equality \cref{(af|_U^V)''=af''|_U^V} follows by additivity, as $\af=\af'+\af''$.
\end{proof}

We proceed by showing that the maps $\af\mapsto\af'$ and $\af\mapsto\af''$ are also compatible with the multiplication in $FI(\cC)$.
\begin{lem}\label{af'-hom-and-af''-anti-hom}
	The map $\af\mapsto\af'$ (resp. $\af\mapsto\af''$) is a homomorphism (resp. anti-homomorphism) of (non-unital) rings $FZ(\cC)\to FZ(\cC)$.
\end{lem}
\begin{proof}
	It was proved in \cref{af'-and-af''-are-finitary-series} that $\af\mapsto\af'$ is additive, so it remains to show that
	\begin{align}\label{af'bt'=(af-bt)'}
		\af'\bt'=(\af\bt)'
	\end{align}
	for all $\af,\bt\in FZ(\cC)$. Clearly, both sides of \cref{af'bt'=(af-bt)'} belong to $FZ(\cC)$. Now, given $x<y$, we have
	\begin{align*}
		(\af'\bt')_{xy}e_{xy}&=\sum_{x<z<y}\af'_{xz}e_{xz}\cdot\bt'_{zy}e_{zy},\\
		(\af\bt)'_{xy}e_{xy}&=((\af\bt)_{xy}e_{xy})'=\sum_{x<z<y}(\af_{xz}e_{xz}\cdot\bt_{zy}e_{zy})',
	\end{align*}
	where the latter equality is explained by \cref{(af_xye_xy)'=af'_xye_xy} and the additivity of the map $\af\mapsto\af'$. Thus, it suffices to prove
that
	\begin{align}\label{af'_xze_xz-bt'_zye_zy=(af_xze_xz-bt_zye_zy)'}
	\af'_{xz}e_{xz}\cdot\bt'_{zy}e_{zy}=(\af_{xz}e_{xz}\cdot\bt_{zy}e_{zy})'
	\end{align}
	for all $x<z<y$. Since $\vf$ is bijective, \cref{af'_xze_xz-bt'_zye_zy=(af_xze_xz-bt_zye_zy)'} is equivalent to
	\begin{align}\label{vf(af'bt')=vf((af-bt)')}
	\vf(\af'_{xz}e_{xz}\cdot\bt'_{zy}e_{zy})=\vf((\af_{xz}e_{xz}\cdot\bt_{zy}e_{zy})').
	\end{align}
	By \cref{vf(rs+sr),af_xye_xy-cdot-bt_uve_uv,defn-of-af'_xy,vf(r^2),vf(e)vf(r)=0} we have
	\begin{align}
		\vf(\af'_{xz}e_{xz}\cdot\bt'_{zy}e_{zy})&=\vf(\af'_{xz}e_{xz})\vf(\bt'_{zy}e_{zy})+\vf(\bt'_{zy}e_{zy})\vf(\af'_{xz}e_{xz})\notag\\
		&\quad-\vf(\bt'_{zy}e_{zy}\cdot\af'_{xz}e_{xz})\notag\\
		&=\vf(e_x)\vf(\af)\vf(e_z)^2\vf(\bt)\vf(e_y)+\vf(e_z)\vf(\bt)\vf(e_y)\vf(e_x)\vf(\af)\vf(e_z)\notag\\
		&=\vf(e_x)\vf(\af)\vf(e_z)\vf(\bt)\vf(e_y).\label{vf(af'_xze_xz-bt'_zye_zy)=vf(e_x)vf(af)vf(e_z)vf(bt)vf(e_y)}
	\end{align}
	Now, by
\cref{defn-of-psi,vf(e_x)psi(af)vf(e_y)=vf(e_x)vf(af)vf(e_y),psi(af_xe-e_xy)=vf(e_x)psi(af)vf(e_y),vf(rs+sr),vf(af_xye_xy)=vf(e_x)vf(af)vf(e_y)+vf(e_y)vf(af)vf(e_x)}
we  get
	\begin{align}
		\vf((\af_{xz}e_{xz}\cdot\bt_{zy}e_{zy})')&=\psi(\af_{xz}e_{xz}\cdot\bt_{zy}e_{zy})\notag\\
		&=\psi((\af_{xz}\bt_{zy})e_{xy})\notag\\
		&=\vf(e_x)\vf((\af_{xz}\bt_{zy})e_{xy})\vf(e_y)\notag\\
		&=\vf(e_x)\vf(\af_{xz}e_{xz}\cdot\bt_{zy}e_{zy})\vf(e_y)\notag\\
		&=\vf(e_x)(\vf(\af_{xz}e_{xz})\vf(\bt_{zy}e_{zy})+\vf(\bt_{zy}e_{zy})\vf(\af_{xz}e_{xz}))\vf(e_y)\notag\\
		&\quad-\vf(e_x)\vf(\bt_{zy}e_{zy}\cdot\af_{xz}e_{xz})\vf(e_y)\notag\\
		&=\vf(e_x)\vf(\af)\vf(e_z)\vf(\bt)\vf(e_y).\label{vf((af_xze_xz-bt_zye_zy)')=vf(e_x)vf(af)vf(e_z)vf(bt)vf(e_y)}
	\end{align}
	Comparing \cref{vf(af'_xze_xz-bt'_zye_zy)=vf(e_x)vf(af)vf(e_z)vf(bt)vf(e_y),vf((af_xze_xz-bt_zye_zy)')=vf(e_x)vf(af)vf(e_z)vf(bt)vf(e_y)}, we get
the desired equality \cref{vf(af'bt')=vf((af-bt)')}.
	
	The proof of the statement about $\af\mapsto\af''$ is analogous.
\end{proof}

 \begin{prop}\label{psi-and-0-hom-and-anti-hom}
 	Let   $\vf:FI(\cC)\to A$ be a Jordan isomorphism and $\psi,\0:FZ(\cC)\to A$ as defined in \cref{defn-of-psi,defn-of-0}.  Then $\psi$ and $\0$ are a
 homomorphism and an anti-homomorphism $FZ(\cC)\to A$, respectively.
 \end{prop}

 \begin{proof} Let $\af,\bt\in FZ(\cC)$ and $x<y$ in $\ob{\cC}$.
 By \cref{vf(e_X)psi(af)=vf(e_X)psi(af|_X),psi(af)vf(e_X)=psi(af|^X)vf(e_X),vf(rs+sr),defn-of-psi,(af|_U^V)'=af'|_U^V}
 \begin{align*}
   \vf(e_x)\psi(\af)\psi(\bt)\vf(e_y)&=\vf(e_x)\psi(\af|_x)\psi(\bt|^y)\vf(e_y)\\
   &=\vf(e_x)\vf((\af|_x)')\vf((\bt|^y)')\vf(e_y)\\
   &=\vf(e_x)\vf(\af'|_x)\vf(\bt'|^y)\vf(e_y)\\
    &=\vf(e_x)\vf(\af'|_x\bt'|^y+\bt'|^y\af'|_x)\vf(e_y)\\
    &\quad-\vf(e_x)\vf(\bt'|^y)\vf(\af'|_x)\vf(e_y)\\
    &=\vf(e_x)\vf(\af'|_x\bt'|^y)\vf(e_y),
 \end{align*}
because $\bt'|^y\af'|_x=0$ and
\begin{align*}
\vf(e_x)\vf(\bt'|^y)\vf(\af'|_x)\vf(e_y)&=\vf(e_x)\vf((\bt|^y)')\vf((\af|_x)')\vf(e_y)\\
	&=\vf(e_x)\psi(\bt|^y)\psi(\af|_x)\vf(e_y)\\
	&=\vf(e_x)\psi(\bt|_x^y)\psi(\af|_x^y)\vf(e_y)\\
	&=\vf(e_x)\psi(\bt_{xy}e_{xy})\psi(\af_{xy}e_{xy})\vf(e_y)\\
	&=\vf(e_x)\psi(\bt)\vf(e_y)\vf(e_x)\psi(\af)\vf(e_y)\\
	&=0
\end{align*}
by \cref{vf(e)vf(r)=0,psi(af_xe-e_xy)=vf(e_x)psi(af)vf(e_y)}. Now using
\cref{properties-of-af|_X^Y,defn-of-af'_xy,vf(e_x)psi(af)vf(e_y)=vf(e_x)vf(af)vf(e_y),af'-hom-and-af''-anti-hom} we get
\begin{align*}
	\vf(e_x)\vf(\af'|_x\bt'|^y)\vf(e_y)&=\vf(e_x)\vf((\af'\bt')|_x^y)\vf(e_y)\\
	&=\vf(e_x)\vf((\af\bt)'|_x^y)\vf(e_y)\\
	&=\vf(e_x)\vf((\af\bt)'_{xy}e_{xy})\vf(e_y)\\
	&=\vf(e_x)\vf(\af\bt)\vf(e_y)\\
	&=\vf(e_x)\psi(\af\bt)\vf(e_y),
\end{align*}
whence
\begin{align*}
	\vf(e_x)\psi(\af)\psi(\bt)\vf(e_y)=\vf(e_x)\psi(\af\bt)\vf(e_y).
\end{align*}
Now, by \cref{vf(e_X)psi(af)=vf(e_X)psi(af|_X),psi(af)vf(e_X)=psi(af|^X)vf(e_X),vf(rs+sr),defn-of-psi,(af|_U^V)'=af'|_U^V}
 \begin{align*}
   \vf(e_y)\psi(\af)\psi(\bt)\vf(e_x)&=\vf(e_y)\psi(\af|_y)\psi(\bt|^x)\vf(e_x)\\
   &=\vf(e_y)\vf((\af|_y)')\vf((\bt|^x)')\vf(e_x)\\
   &=\vf(e_y)\vf(\af'|_y)\vf(\bt'|^x)\vf(e_x)\\
    &=\vf(e_y)\vf(\af'|_y\bt'|^x+\bt'|^x\af'|_y)\vf(e_x)\\
    &\quad-\vf(e_y)\vf(\bt'|^x)\vf(\af'|_y)\vf(e_x)\\
    &=-\vf(e_y)\vf(\bt'|^x)\vf(\af'|_y)\vf(e_x),
 \end{align*}
because $\af'|_y\bt'|^x= \bt'|^x\af'|_y=0$, by \cref{properties-of-af|_X^Y}. But
\begin{align*}
\vf(e_y)\vf(\bt'|^x)\vf(\af'|_y)\vf(e_x)&=\vf(e_y)\vf((\bt|^x)')\vf((\af|_y)')\vf(e_x)\\
	&=\vf(e_y)\psi(\bt|^x)\psi(\af|_y)\vf(e_x)\\
	&=\vf(e_y)\psi(\bt|_y^x)\psi(\af|_y^x)\vf(e_x)\\
	&=0,
\end{align*}
because $\bt|_y^x=\af|_y^x=0$. On the other hand, $\vf(e_y)\psi(\af\bt)\vf(e_x)=0$, by
\cref{vf(e_y)psi(af)vf(e_x)=vf(e_x)0(af)vf(e_y)=vf(e_x)psi(af)vf(e_x)=vf(e_x)0(af)vf(e_x)=0}.
Hence
\begin{align*}
	\vf(e_x)\psi(\af)\psi(\bt)\vf(e_y)+\vf(e_y)\psi(\af)\psi(\bt)\vf(e_x)= \vf(e_x)\psi(\af\bt)\vf(e_y)+\vf(e_y)\psi(\af\bt)\vf(e_x).
\end{align*}

Finally, consider $x=y$. By
\cref{vf(e_X)psi(af)=vf(e_X)psi(af|_X),psi(af)vf(e_X)=psi(af|^X)vf(e_X),vf(rs+sr),defn-of-psi,properties-of-af|_X^Y,(af|_U^V)'=af'|_U^V}
 \begin{align*}
   \vf(e_x)\psi(\af)\psi(\bt)\vf(e_x)&=\vf(e_x)\psi(\af|_x)\psi(\bt|^x)\vf(e_x)\\
   &=\vf(e_x)\vf((\af|_x)')\vf((\bt|^x)')\vf(e_x)\\
   &=\vf(e_x)\vf(\af'|_x)\vf(\bt'|^x)\vf(e_x)\\
    &=\vf(e_x)\vf(\af'|_x\bt'|^x+\bt'|^x\af'|_x)\vf(e_x)\\
    &\quad-\vf(e_x)\vf(\bt'|^x)\vf(\af'|_x)\vf(e_x)\\
    &= \vf(e_x)\vf(\bt'|^x\af'|_x)\vf(e_x)\\
    &\quad-\vf(e_x)\vf(\bt'|^x)\vf(\af'|_x)\vf(e_x).
 \end{align*}
By \cref{vf(af_xxe_x)=vf(e_x)vf(af)vf(e_x)}, $\vf(e_x)\vf(\bt'|^x\af'|_x)\vf(e_x)=  \vf((\bt'|^x\af'|_x)_{xx}e_x)=0$, because $\bt'|^x\af'|_x\in
FZ(\cC)$ and using \cref{vf(e_X)psi(af)=vf(e_X)psi(af|_X),psi(af)vf(e_X)=psi(af|^X)vf(e_X),defn-of-psi,(af|_U^V)'=af'|_U^V}
\begin{align*}
\vf(e_x)\vf(\bt'|^x)\vf(\af'|_x)\vf(e_x)&=\vf(e_x)\vf((\bt|^x)')\vf((\af|_x)')\vf(e_x)\\
	&=\vf(e_x)\psi(\bt|^x)\psi(\af|_x)\vf(e_x)\\
	&=\vf(e_x)\psi(\bt|_x^x)\psi(\af|_x^x)\vf(e_x)\\
	&=0,
\end{align*}
because $\bt|_x^x=\af|_x^x = 0$.  We also have $\vf(e_x)\psi(\af\bt)\vf(e_x)=0$, by
\cref{vf(e_y)psi(af)vf(e_x)=vf(e_x)0(af)vf(e_y)=vf(e_x)psi(af)vf(e_x)=vf(e_x)0(af)vf(e_x)=0}. It follows that
\begin{align*}
	\vf(e_x)\psi(\af)\psi(\bt)\vf(e_x)=\vf(e_x)\psi(\af\bt)\vf(e_x).
\end{align*}
Therefore, by \cref{a=b-in-A}, we conclude that $\psi(\af \bt)= \psi(\af)\psi(\bt)$.

In an analogous way one proves that $\0$ is an anti-homomorphism $FZ(\cC)\to A$.
\end{proof}

\subsection{Decomposition of $\vf$ into a near-sum}

The idea now is to extend $\psi$ and $\theta$ to the whole $FI(\cC)$ in order to obtain a decomposition of $\varphi$ into a near-sum.

\begin{lem}\label{a-in-D(cC)-bt-in-FZ(cC)}
	Let $\vf:FI(\cC)\to A$ be a Jordan isomorphism and $\psi:FZ(\cC)\to A$ as defined at \cref{defn-of-psi}.  If $\alpha\in D(\cC)$ and $\beta \in
FZ(\cC)$, then
	\begin{align}\label{psi(af.bt)=vf(af)psi(bt)}
		\psi(\alpha\beta)=\vf(\alpha)\psi(\beta).
	\end{align}
\end{lem}

\begin{proof}
	We will use  \cref{a=b-in-A}. Let $x<y$. Since $\alpha\beta\in FZ(\cC)$, we have
	\begin{align}
		\vf(e_x)\psi(\af\bt)\vf(e_y)+\vf(e_y)\psi(\af\bt)\vf(e_x) & =
\vf(e_x)\vf(\af\bt)\vf(e_y)\label{vf(e_x)psi(af-bt)vf(e_y)+vf(e_y)psi(af-bt)vf(e_x)=vf(e_x)vf(af-bt)vf(e_y)} \\
		&=\vf(e_x)\vf((\af\bt)_{xy}e_{xy})\vf(e_y)\label{vf(e_x)vf(af-bt)vf(e_y)=vf(e_x)vf((af-bt)_xye_xy)vf(e_y)} \\
&=\vf(e_x)\vf(\af_{xx}\bt_{xy}e_{xy})\vf(e_y), \label{vf(e_x)psi(af bt)vf(e_y)+vf(e_y)psi(afbt)vf(e_x)}
	\end{align}
	where equality \cref{vf(e_x)psi(af-bt)vf(e_y)+vf(e_y)psi(af-bt)vf(e_x)=vf(e_x)vf(af-bt)vf(e_y)} follows from
\cref{vf(e_x)psi(af)vf(e_y)=vf(e_x)vf(af)vf(e_y),vf(e_y)psi(af)vf(e_x)=vf(e_x)0(af)vf(e_y)=vf(e_x)psi(af)vf(e_x)=vf(e_x)0(af)vf(e_x)=0}, equality
\cref{vf(e_x)vf(af-bt)vf(e_y)=vf(e_x)vf((af-bt)_xye_xy)vf(e_y)} follows from \cref{vf(e_x)vf(af)vf(e_y)=vf(e_x)vf(af_xye_xy)vf(e_y)}, and
\cref{vf(e_x)psi(af bt)vf(e_y)+vf(e_y)psi(afbt)vf(e_x)} uses the fact that $\af\in D(\cC)$.
	
	On the other hand, since $e_y\af=\af e_y$, by \cref{vf(e)vf(r)=vf(r)vf(e)} we have
	\begin{equation}
	\vf(e_y)\vf(\af)\psi(\bt)\vf(e_x)=\vf(\af)\vf(e_y)\psi(\bt)\vf(e_x)=0, \label{vf(e_y)vf(af)psi(bt)vf(e_x)=vf(af)vf(e_y)psi(bt)vf(e_x)=0}
	\end{equation}
	where the last equality follows from \cref{vf(e_y)psi(af)vf(e_x)=vf(e_x)0(af)vf(e_y)=vf(e_x)psi(af)vf(e_x)=vf(e_x)0(af)vf(e_x)=0}.  Furthermore, by
	\cref{vf(r^2),vf(e)vf(r)=vf(r)vf(e),vf(af_xxe_x)=vf(e_x)vf(af)vf(e_x),vf(e_x)psi(af)vf(e_y)=vf(e_x)vf(af)vf(e_y)}
	\begin{align}
		\vf(e_x)\vf(\af)\psi(\bt)\vf(e_y) &= \vf(e_x) \vf(\af) \vf(e_x) \psi(\bt)\vf(e_y)=\vf(e_x)\vf(\af_{xx}e_{xx})
\vf(\bt)\vf(e_y).\label{vf(e_x)psi(af)psi(bt)vf(e_y)=vf(af_xx.e_xx)vf(bt)vf(e_y)}
	\end{align}
	Now by \cref{vf(rs+sr)}
	\begin{align*}
		\vf(\af_{xx}e_{xx})\vf(\bt)=\vf(\af_{xx}e_{xx}\bt+\bt\af_{xx}e_{xx})-\vf(\bt)\vf(\af_{xx}e_{xx}).
	\end{align*}
	Since $\vf(\af_{xx}e_{xx})\vf(e_y)=0$ in view of \cref{vf(e)vf(r)=0}, we obtain from
\cref{vf(e_x)psi(af)psi(bt)vf(e_y)=vf(af_xx.e_xx)vf(bt)vf(e_y),vf(e_x)psi(af)vf(e_y)=vf(e_x)vf(af)vf(e_y),vf(e_X)psi(af)vf(e_Y)=vf(e_X)psi(af|_X^Y)vf(e_Y)}
	\begin{align*}
		\vf(e_x)\vf(\af)\psi(\bt)\vf(e_y)&=\vf(e_x)\vf(\af_{xx}e_{xx}\bt+\bt\af_{xx}e_{xx})\vf(e_y)\\
		&=\vf(e_x)\psi(\af_{xx}e_{xx}\bt+\bt\af_{xx}e_{xx})\vf(e_y)\\
		&=\vf(e_x)\psi((\af_{xx}e_{xx}\bt+\bt\af_{xx}e_{xx})|_x^y)\vf(e_y)\\
		&=\vf(e_x)\psi(\af_{xx}\bt_{xy}e_{xy})\vf(e_y)\\
		&=\vf(e_x)\vf(\af_{xx}\bt_{xy}e_{xy})\vf(e_y).
	\end{align*}
	Combining this with \cref{vf(e_x)psi(af bt)vf(e_y)+vf(e_y)psi(afbt)vf(e_x),vf(e_y)vf(af)psi(bt)vf(e_x)=vf(af)vf(e_y)psi(bt)vf(e_x)=0} we have
	$$
	\vf(e_x)\psi(\af\bt)\vf(e_y)+\vf(e_y)\psi(\af\bt)\vf(e_x)= \vf(e_x)\vf(\af)\psi(\bt)\vf(e_y) + \vf(e_y)\vf(\af)\psi(\bt)\vf(e_x).
	$$
	Moreover, for any $x\in \ob{\cC}$
	$$
	\vf(e_x)\vf(\af)\psi(\bt)\vf(e_x) = \vf(\af)\vf(e_x)\psi(\bt)\vf(e_x)=0.
	$$
	The last equality follows from \cref{vf(e_y)psi(af)vf(e_x)=vf(e_x)0(af)vf(e_y)=vf(e_x)psi(af)vf(e_x)=vf(e_x)0(af)vf(e_x)=0}, as well as the equality
$\vf(e_x)\psi(\af \bt)\vf(e_x)=0$. Hence \cref{psi(af.bt)=vf(af)psi(bt)} holds by \cref{a=b-in-A}.
\end{proof}

\begin{rem}\label{af-in-FZ(cC)-bt-in-D(cC)}
	By a similar computation one proves that if $\alpha\in FZ(\cC)$ and $\beta \in D(\cC)$, then $\psi(\alpha\beta)=\psi(\alpha)\vf(\beta)$.
\end{rem}

Let us now extend $\psi$ and $\0$ to the additive maps $\tl\psi$ and $\tl\0$ defined on the whole ring $FI(\cC)$ by means of
\begin{align}
	\tl\psi(\af)&=\vf(\af_D)+\psi(\af_Z),\label{tl-psi(af)=vf(af_D)+psi(af_Z)}\\
	\tl\0(\af)&=\vf(\af_D)+\0(\af_Z).\label{tl-0(af)=vf(af_D)+0(af_Z)}
\end{align}

The following theorem is the main result of \cref{sec-jiso-FI(C)}.

\begin{thrm}\label{vf-near-sum-of-tl-psi-and-tl-theta}
	Let $\vf:FI(\cC)\to A$ be a Jordan isomorphism and $\tl\psi,\tl\theta:FI(\cC)\to A$ as defined in
\cref{tl-psi(af)=vf(af_D)+psi(af_Z),tl-0(af)=vf(af_D)+0(af_Z)}. Then $\vf$ is the near sum of $\tl\psi$ and $\tl\0$ with respect to $D(\cC)$ and
$FZ(\cC)$. Moreover,
	\begin{enumerate}
		\item $\tl\psi$ is a homomorphism if and only if $\vf|_{D(\cC)}$ is a homomorphism;\label{tl-psi-homo<=>vf|_D-homo}
		\item $\tl\theta$ is an anti-homomorphism if and only if $\vf|_{D(\cC)}$ is an anti-homomorphism.\label{tl-0-homo<=>vf|_D-anti-homo}
	\end{enumerate}
	In particular, if both \cref{tl-psi-homo<=>vf|_D-homo,tl-0-homo<=>vf|_D-anti-homo} are true, then $D(\cC)$ is a commutative ring.
\end{thrm}
\begin{proof}
Clearly  $\tl\psi$ and $\tl\0$ are additive maps such that $\vf|_{D(\cC)}= \tl\psi|_{D(\cC)}=\tl\0|_{D(\cC)}$ and, by \cref{vf|_FZ-is-sum},
$\vf|_{FZ(\cC)}=\tl\psi|_{FZ(\cC)}+\tl\0|_{FZ(\cC)}$.

It remains to prove that, given $\af,\bt\in FZ(\cC)$, the products $\tl\psi(\af)\tl\0(\bt)$ and $\tl\0(\bt)\tl\psi(\af)$ are zero. We shall show that
$\tl\psi(\af)\tl\0(\bt)=0$, leaving the proof of $\tl\0(\bt)\tl\psi(\af)=0$, which is analogous, to the reader.

Let $\af,\bt\in FZ(\cC)$. Then $\tl\psi(\af)=\psi(\af)$ and $\tl\0(\bt)=\0(\bt)$. So we need to prove that $\psi(\af)\0(\bt)=0$.  Our main tool will be
\cref{a=b-in-A}. Taking an arbitrary $x\in\ob{\cC}$, we have by
\cref{vf(e_X)psi(af)=vf(e_X)psi(af|_X),0(af)vf(e_X)=0(af|_X)vf(e_X),defn-of-psi,defn-of-0}
\begin{align}\label{vf(e_x)psi(af)0(bt)vf(e_x)=vf(e_x)psi(af|_x)0(bt|_x)vf(e_x)}
	\vf(e_x)\psi(\af)\0(\bt)\vf(e_x)=\vf(e_x)\psi(\af|_x)\0(\bt|_x)\vf(e_x)=\vf(e_x)\vf((\af|_x)')\vf((\bt|_x)'')\vf(e_x).
\end{align}
But
\begin{align*}
	\vf((\af|_x)')\vf((\bt|_x)'')=\vf((\af|_x)'(\bt|_x)''+(\bt|_x)''(\af|_x)')-\vf((\bt|_x)'')\vf((\af|_x)')
\end{align*}
thanks to \cref{vf(rs+sr)}, and
\begin{align*}
	\vf(e_x)\vf((\af|_x)'(\bt|_x)'')\vf(e_x)= 0 = \vf(e_x)\vf((\bt|_x)''(\af|_x)')\vf(e_x)
\end{align*}
in view of \cref{vf(af_xxe_x)=vf(e_x)vf(af)vf(e_x)} and the fact that $(\af|_x)',(\bt|_x)''\in FZ(\cC)$. Therefore,
\cref{vf(e_x)psi(af)0(bt)vf(e_x)=vf(e_x)psi(af|_x)0(bt|_x)vf(e_x)} equals
\begin{align*}
	-\vf(e_x)\vf((\bt|_x)'')\vf((\af|_x)')\vf(e_x)&=-\vf(e_x)\0(\bt|_x)\psi(\af|_x)\vf(e_x)\\
	&=-\vf(e_x)\0(\bt|_x^x)\psi(\af|_x^x)\vf(e_x)\\
	&=0.
\end{align*}
Here we have used \cref{defn-of-psi,defn-of-0,vf(e_X)0(af)=vf(e_X)0(af|^X),psi(af)vf(e_X)=psi(af|^X)vf(e_X)} and the easy observations that
$\af|_x^x=\af_{xx}e_{xx}=0$ and similarly $\bt|_x^x=0$. Now take $x<y$ and notice from
\cref{vf(rs+sr),vf(e_X)psi(af)=vf(e_X)psi(af|_X),0(af)vf(e_X)=0(af|_X)vf(e_X)} that
\begin{align}
	\vf(e_x)\psi(\af)\0(\bt)\vf(e_y)&=\vf(e_x)\psi(\af|_x)\0(\bt|_y)\vf(e_y)\notag\\
	&=\vf(e_x)\vf((\af|_x)')\vf((\bt|_y)'')\vf(e_y)\notag\\
	&=\vf(e_x)\vf((\af|_x)'(\bt|_y)''+(\bt|_y)''(\af|_x)')\vf(e_y)\label{vf(e_x)vf((af|_x)'(bt|_y)''+(bt|_y)''(af|_x)')vf(e_y)}\\
	&\quad-\vf(e_x)\vf((\bt|_y)'')\vf((\af|_x)')\vf(e_y).\label{-vf(e_x)vf((bt|_y)'')vf((af|_x)')vf(e_y)}
\end{align}
The summand \cref{-vf(e_x)vf((bt|_y)'')vf((af|_x)')vf(e_y)} is
\begin{align*}
	-\vf(e_x)\0(\bt|_y)\psi(\af|_x)\vf(e_y)=-\vf(e_x)\0(\bt|_y^x)\psi(\af|_x^y)\vf(e_y)=0
\end{align*}
in view of \cref{defn-of-psi,defn-of-0,vf(e_X)0(af)=vf(e_X)0(af|^X),psi(af)vf(e_X)=psi(af|^X)vf(e_X)} and the fact that $\bt|_y^x=0$. Furthermore, use
\cref{(af-bt)|_X^Y=af|_X-bt|^Y,(af+bt)|_X^Y=af|_X^Y+bt|_X^Y,(af|_X^Y)_U^V} of \cref{properties-of-af|_X^Y},
\cref{af=af'+af'',vf(e_x)psi(af)vf(e_y)=vf(e_x)vf(af)vf(e_y),vf(e_X)psi(af)vf(e_Y)=vf(e_X)psi(af|_X^Y)vf(e_Y)} to transform
\cref{vf(e_x)vf((af|_x)'(bt|_y)''+(bt|_y)''(af|_x)')vf(e_y)} into
\begin{align*}
	&\vf(e_x)\psi((\af|_x)'(\bt|_y)''+(\bt|_y)''(\af|_x)')\vf(e_y)\\
	&\quad=\vf(e_x)\psi(((\af|_x)'(\bt|_y)''+(\bt|_y)''(\af|_x)')|_x^y)\vf(e_y)\\
	&\quad=\vf(e_x)\psi((\af|_x)'|_x(\bt|_y)''|^y+(\bt|_y)''|_x(\af|_x)'|^y)\vf(e_y)\\
	&\quad=\vf(e_x)\psi((\af|_x)'|_x(\bt|_y-(\bt|_y)')|^y+(\bt|_y-(\bt|_y)')|_x(\af|_x)'|^y)\vf(e_y)\\
	&\quad=-\vf(e_x)\psi((\af|_x)'|_x(\bt|_y)'|^y+(\bt|_y)'|_x(\af|_x)'|^y)\vf(e_y)\\
	&\quad=-\vf(e_x)\psi(((\af|_x)'(\bt|_y)'+(\bt|_y)'(\af|_x)')|_x^y)\vf(e_y)\\
	&\quad=-\vf(e_x)\psi((\af|_x)'(\bt|_y)'+(\bt|_y)'(\af|_x)'))\vf(e_y)\\
	&\quad=-\vf(e_x)\vf((\af|_x)'(\bt|_y)'+(\bt|_y)'(\af|_x)')\vf(e_y).
\end{align*}
In view of \cref{vf(rs+sr),defn-of-psi,defn-of-0,vf(e_X)psi(af)=vf(e_X)psi(af|_X),psi(af)vf(e_X)=psi(af|^X)vf(e_X)} the latter is
\begin{align*}
	&-\vf(e_x)\vf((\af|_x)')\vf((\bt|_y)')\vf(e_y)-\vf(e_x)\vf((\bt|_y)')\vf((\af|_x)')\vf(e_y)\\
	&\quad=-\vf(e_x)\psi(\af|_x)\psi(\bt|_y)\vf(e_y)-\vf(e_x)\psi(\bt|_y)\psi(\af|_x)\vf(e_y)\\
	&\quad=-\vf(e_x)\psi(\af|_x)\psi(\bt|_y^y)\vf(e_y)-\vf(e_x)\psi((\bt|_y)|_x)\psi(\af|_x)\vf(e_y)\\
	&\quad=0,
\end{align*}
as $\bt|_y^y = 0 =(\bt|_y)|_x$. Consequently, $\vf(e_x)\psi(\af)\0(\bt)\vf(e_y)=0$. The argument that proves $\vf(e_y)\psi(\af)\0(\bt)\vf(e_x)=0$ is
totally symmetric (just switch $x$ and $y$ in the proof above). Thus,
\begin{align*}
	\vf(e_x)\psi(\af)\0(\bt)\vf(e_y)+\vf(e_y)\psi(\af)\0(\bt)\vf(e_x)=0
\end{align*}
for all $x<y$. By \cref{a=b-in-A} the product $\psi(\af)\0(\bt)$ is zero.
	
For the second statement of the theorem take $\af,\bt\in FI(\cC)$. Write $\af=\af_D+\af_Z$ and $\bt=\bt_D+\bt_Z$. Then
$$
\af\bt=\af_D\bt_D + \af_D\bt_Z + \af_Z\bt_D + \af_Z\bt_Z.
$$
Since $D(\cC)$ is a subring of $FI(\cC)$ and $FZ(\cC)$ is an ideal of $FI(\cC)$, we have that $\af_D\bt_D\in D(\cC)$  and
$\af_D\bt_Z,\af_Z\bt_D,\af_Z\bt_Z\in FZ(\cC)$. If $\tl\psi$ is a homomorphism, then, obviously, $\tl\psi|_{D(\cC)}=\vf|_{D(\cC)}$ is a homomorphism.
Conversely, combining the fact that $\tl\psi|_{D(\cC)}=\vf|_{D(\cC)}$ is a homomorphism with
\cref{psi-and-0-hom-and-anti-hom,a-in-D(cC)-bt-in-FZ(cC),af-in-FZ(cC)-bt-in-D(cC)}, one can show that $\tl\psi(\af\bt)=\tl\psi(\af)\tl\psi(\bt)$, that
is, $\tl\psi$ is a homomorphism. This proves \cref{tl-psi-homo<=>vf|_D-homo}. The proof of \cref{tl-0-homo<=>vf|_D-anti-homo} is similar.
\end{proof}

As a consequence, we obtain~\cite[Theorem~3.13]{BFK} without using the results of~\cite{Akkurts-Barker} and thus without the restriction that $R$ is $2$-torsionfree.
\begin{cor}\label{vf-near-sum-for-FI(P.R)}
Let $P$ be a poset, $R$ a commutative ring and $A$ an $R$-algebra. Then each $R$-linear Jordan isomorphism $\vf:FI(P,R)\to A$ is the near sum of a homomorphism
and an anti-homomorphism with respect to $D(P,R)$ and $FZ(P,R)$.
\end{cor}
\begin{proof}
Indeed, it was proved in~\cite[Proposition 3.3]{BFK} that $\vf|_{D(P,R)}$ is a homomorphism and an anti-homomorphism at the same time.
\end{proof}

\section{Jordan isomorphisms of $FI(P,R)$}\label{jord-iso-FI(P_R)}

Observe that the condition that $\vf|_{D(\cC)}$ is a homomorphism or an anti-homo\-morphism from \cref{vf-near-sum-of-tl-psi-and-tl-theta} may fail for $\cC=\cC(P,R)$, where $P$ is a quasiordered set, which is not a poset and $R$ is a commutative ring. Indeed, suppose that $1<|P|<\infty$ and $x\le y$ for all $x,y\in P$, so that $P=\bar x=\{y\in P\mid y\sim x\}$ for an arbitrary fixed $x\in P$. In this case $FI(\cC)$ coincides with $D(\cC)$ and is isomorphic to the full matrix ring $M_n(R)$, where $n=|P|$. If $R$ has a non-trivial idempotent $e$, then the map $J(A)=eA+(1-e)A^T$, where $A^T$ is the transpose of $A$, is a Jordan automorphism of $M_n(R)$, which is neither a homomorphism, nor an anti-homomorphism. This is a particular case of the example given in the introduction of \cite{BeBreChe}. Notice also that in this case $J$ is the sum of a homomorphism and an anti-homomorphism, which is true for an arbitrary Jordan homomorphism of $M_n(R)$ by \cite[Theorem 7]{Jacobson-Rickart50}.

The above example shows that it would be natural to find some sufficient conditions under which a Jordan isomorphism $\vf:FI(\cC)\to A$ could be decomposed as the sum of a homomorphism and an anti-homomorphism. Our final aim will be to prove the existence of such a decomposition in the case $\cC=\cC(P,R)$, but we start with the results which hold in the general situation.

\subsection{Decomposition of $\vf|_{D(\cC)}$}

Since we already know by \cref{psi-and-0-hom-and-anti-hom,vf|_FZ-is-sum} that $\vf|_{FZ(\cC)}=\psi+\0$, where $\psi$ is a homomorphism and $\0$ is an anti-homomorphism, our first goal will be to find (under certain conditions) a similar decomposition for $\vf|_{D(\cC)}$.

For each $x\in\ob{\cC}$ we introduce the following notations
\begin{align*}
	D(\cC)_x&=\{\af_{xx}e_{xx}\mid \af_{xx}\in\mor xx\}\subseteq D(\cC),\\
	\vf_x&=\vf|_{D(\cC)_x}:D(\cC)_x\to \vf(D(\cC)_x).
\end{align*}
Observe that $D(\cC)_x$ is a ring with identity $e_x$, $D(\cC)_x\cong\mor xx$ and $D(\cC)\cong\prod_{x\in\ob{\cC}}D(\cC)_x$.

\begin{lem}\label{vf(eRe)-subring}
	Let $\vf:R\to S$ be a Jordan isomorphism of associative rings. Then for any idempotent $e\in R$, the set $\vf(eRe)$ is a ring under the operations of $S$.
\end{lem}
\begin{proof}
	Clearly, $\vf(eRe)$ is a subgroup of the additive group of $S$. Let $r,s\in eRe$. Since $\vf$ is surjective, there exists $t\in R$ such that
	\begin{align}\label{vf(r)vf(s)=vf(t)}
	\vf(r)\vf(s)=\vf(t).
	\end{align}
	The idempotent $e$ is the identity of $eRe$, so $\vf(e)\vf(r)=\vf(r)$ and $\vf(s)\vf(e)=\vf(s)$ by \cref{vf(e)vf(r)=vf(r)vf(e)}. Therefore, in view of \cref{vf(r)vf(s)=vf(t),vf(rsr)},
	\begin{align*}
	\vf(ete)=\vf(e)\vf(t)\vf(e)=\vf(e)\vf(r)\vf(s)\vf(e)=\vf(r)\vf(s)=\vf(t),
	\end{align*}
	whence $t=ete$ thanks to the injectivity of $\vf$. Thus, $\vf(eRe)$ is closed under the multiplication in $S$, so it is a ring with identity $\vf(e)$.
\end{proof}

\begin{cor}\label{vf(D(C)_x)-subring}
	Let $\vf:FI(\cC)\to A$ be a Jordan isomorphism. Then for each $x\in\ob\cC$ the set $\vf(D(\cC)_x)$ is a ring under the operations of $A$.
\end{cor}
\begin{proof}
	Indeed, $D(\cC)_x=e_xFI(\cC)e_x$.
\end{proof}

\begin{lem}\label{vf|_D(C)-is-the-sum-of-psi-and-theta}
Let $\vf: FI(\cC)\to A$ be a Jordan isomorphism. Then $\vf|_{D(\cC)}:D(\cC)\to A$ is the sum of a homomorphism and an anti-homomorphism if and only if for each $x\in\ob{\cC}$ the restriction $\vf_x:D(\cC)_x\to \vf(D(\cC)_x)$ is the sum of a homomorphism and an anti-homomorphism.
\end{lem}
\begin{proof}
The ``only if'' part is obvious, so we only need to prove the ``if'' part. Let
\begin{align}\label{vf_x=psi_x+0_x}
	\vf_x=\psi_x + \0_x
\end{align}
be the decomposition of $\vf_x$, where $\psi_x,\0_x:D(\cC)_x\to \vf(D(\cC)_x)$ are a homomorphism and an anti-homomorphism, respectively. For $\af=\sum_{x\in \ob{\cC}}\af_{xx}e_{xx}\in D(\cC)$ we have   $\psi_{x}(\af_{xx}e_{xx})\in \vf(D(\cC)_x)$ for each $x\in \ob{\cC}$, so there exists $\tl\af_{xx}e_{xx}\in
D(\cC)_x$ such that
\begin{align}\label{psi_x(af_xx.e_xx)=vf(tl-af_xx.e_xx)}
	\psi_{x}(\af_{xx}e_{xx})=\vf(\tl\af_{xx}e_{xx}).
\end{align}
Similarly, for each $x \in \ob{\cC}$, there exists $\tl{\tl{\af}}_{xx}e_{xx}\in D(\cC)_x$ such that
\begin{align}\label{0_x(af_xx.e_xx)=vf(tl-tl-af_xx.e_xx)}
	\0_{x}(\af_{xx}e_{xx})=\vf(\tl{\tl{\af}}_{xx}e_{xx}).
\end{align}
Let
\begin{align*}%\label{tl-af-and-tl-tl-af}
	\tl\af= \sum_{x\in \ob{\cC}}\tl\af_{xx}e_{xx} \ \ \ \  \text{and} \  \ \ \ \tl{\tl{\af}}= \sum_{x\in \ob{\cC}}\tl{\tl{\af}}_{xx}e_{xx}
\end{align*}
and define $\psi, \0:D(\cC)\to A$ by
\begin{align}\label{psi(af)=vf(tl-af)-and-0(af)=vf(tl-tl-af)}
	\psi(\af)=\vf(\tl\af)  \ \ \ \ \text{and}  \ \ \ \ \0(\af)=\vf(\tl{\tl{\af}}).
\end{align}
For each $x\in \ob{\cC}$, by \cref{vf_x=psi_x+0_x,psi_x(af_xx.e_xx)=vf(tl-af_xx.e_xx),0_x(af_xx.e_xx)=vf(tl-tl-af_xx.e_xx)}
\begin{align*}
\vf(\af_{xx}e_{xx}) & =\vf_{x}(\af_{xx}e_{xx}) =\psi_{x}(\af_{xx}e_{xx})+\0_{x}(\af_{xx}e_{xx})\\
               & = \vf(\tl\af_{xx}e_{xx})+\vf(\tl{\tl{\af}}_{xx}e_{xx})\\
               & = \vf((\tl\af_{xx}  + \tl{\tl{\af}}_{xx})e_{xx}).
\end{align*}
Since $\vf$ is injective, $\af_{xx}e_{xx}=(\tl\af_{xx}+\tl{\tl{\af}}_{xx})e_{xx}$, for each $x\in \ob{\cC}$. Hence,
\begin{align*}%\label{af=tl-af+tl-tl-af}
	\af=\tl\af +\tl{\tl{\af}},
\end{align*}
and consequently
$$
\vf(\af)=\vf(\tl\af + \tl{\tl{\af}})=\psi(\af)+\0(\af)=(\psi+\0)(\af)
$$
in view of \cref{psi(af)=vf(tl-af)-and-0(af)=vf(tl-tl-af)}. Thus $\vf|_{D(\cC)}=\psi+\0$.

Now, we show that $\psi$ is a homomorphism. Let us first prove that $\psi$ is additive. Take $$\af=\sum_{x\in \ob{\cC}}\af_{xx}e_{xx}, \ \ \bt=\sum_{x\in \ob{\cC}}\bt_{xx}e_{xx} \in  D(\cC).$$ For each $x\in \ob{\cC}$, using \cref{psi_x(af_xx.e_xx)=vf(tl-af_xx.e_xx)}, we have
\begin{align*}
\vf((\tl\af_{xx}+\tl\bt_{xx})e_{xx}) & =\vf(\tl\af_{xx}e_{xx})+ \vf(\tl\bt_{xx}e_{xx})=\psi_{x}(\af_{xx}e_{xx})+ \psi_{x}(\bt_{xx}e_{xx})\\
 & = \psi_{x}((\af+\bt)_{xx}e_{xx})=\vf(\widetilde{(\af+\bt)}_{xx}e_{xx}),
\end{align*}
 so $\tl\af+\tl\bt=\widetilde{\af+\bt}$. It follows by \cref{psi(af)=vf(tl-af)-and-0(af)=vf(tl-tl-af)} that
\begin{align*}
  \psi(\af+\bt) & =\vf(\widetilde{\af+\bt}) = \vf(\tl\af +\tl\bt)=\vf(\tl\af)+\vf(\tl\bt)=\psi(\af)+\psi(\bt).
\end{align*}
In order to show that $\psi(\af\bt)=\psi(\af)\psi(\bt)$, we will use \cref{a=b-in-A}. Let $u\in \ob{\cC}$. By \cref{vf(rsr),e_x-alpha-e_y,psi(af)=vf(tl-af)-and-0(af)=vf(tl-tl-af)} we have
\begin{align*}
  \vf(e_u)\psi(\af\bt)\vf(e_u)  & = \vf(e_u)\vf(\widetilde{\af\bt})\vf(e_u)= \vf(e_u\widetilde{\af\bt}e_u) = \vf((\widetilde{\af\bt})_{uu} e_{uu}) \\
   & = \psi_u((\af\bt)_{uu}e_{uu})= \psi_u(\af_{uu}e_{uu})\psi_u(\bt_{uu}e_{uu}) = \vf(\tl\af_{uu}e_{uu})\vf(\tl\bt_{uu}e_{uu})\\
   & = \vf(e_u\tl\af e_u)\vf(e_u\tl\bt e_u) = \vf(e_u)\vf(\tl\af)\vf(e_u)\vf(e_u)\vf(\tl\bt)\vf(e_u).
\end{align*}
In view of the fact that $e_u$ is a central idempotent of $D(\cC)$ and \cref{vf(e)vf(r)=vf(r)vf(e)}, the last product equals
\begin{align*}
\vf(e_u)^2\vf(\tl\af)\vf(\tl\bt)\vf(e_u)^2 = \vf(e_u)\vf(\tl\af)\vf(\tl\bt)\vf(e_u) = \vf(e_u)\psi(\af)\psi(\bt)\vf(e_u).
  \end{align*}
Now, consider $u,v\in \ob{\cC}$, $u<v$. Since the central idempotents $e_u$ and $e_v$ are orthogonal, equalities \cref{vf(e)vf(r)=0,psi(af)=vf(tl-af)-and-0(af)=vf(tl-tl-af)} imply
\begin{align*}
   \vf(e_u)\psi(\af\bt)\vf(e_v)+ \vf(e_v)\psi(\af\bt)\vf(e_u)  & = \vf(e_u)\vf(\widetilde{\af\bt})\vf(e_v)+\vf(e_v)\vf(\widetilde{\af\bt})\vf(e_u)\\
  & = \vf(e_u)\vf(e_v)\vf(\widetilde{\af\bt})+\vf(e_v)\vf(e_u)\vf(\widetilde{\af\bt}) = 0.
  \end{align*}
Similarly,
\begin{align*}
& \vf(e_u)\psi(\af) \psi(\bt)\vf(e_v)+ \vf(e_v)\psi(\af)\psi(\bt)\vf(e_u) \\
     & = \vf(e_u)\vf(\tl\af)\vf(\tl\bt)\vf(e_v)+\vf(e_v)\vf(\tl\af)\vf(\tl\bt)\vf(e_u)=0.
\end{align*}
It follows that
$$ \vf(e_u)\psi(\af\bt)\vf(e_v)+ \vf(e_v)\psi(\af\bt)\vf(e_u) = \vf(e_u)\psi(\af) \psi(\bt)\vf(e_v)+ \vf(e_v)\psi(\af)\psi(\bt)\vf(e_u).$$
By \cref{a=b-in-A}, $\psi(\af\bt)=\psi(\af)\psi(\bt)$, and therefore $\psi$ is a homomorphism.

The proof that $\theta$ is an anti-homomorphism is analogous.
\end{proof}

\subsection{Decomposition of $\vf$ into a sum}

Let $\vf: FI(\cC)\to A$ be a Jordan isomorphism and write $\vf|_{FZ(\cC)}=\psi_Z+\0_Z$, where $\psi_Z$ and $\0_Z$ are given by \cref{defn-of-psi,defn-of-0}. Assume also that for all $x\in\ob{\cC}$ the map $\vf_x$ is the sum \cref{vf_x=psi_x+0_x} of a homomorphism $\psi_x$ and an anti-homomorphism $\0_x$, and let $\vf|_{D(\cC)}=\psi_D+\0_D$ be the corresponding decomposition of $\vf|_{D(\cC)}$ constructed in \cref{vf|_D(C)-is-the-sum-of-psi-and-theta}. Define $\psi, \0:FI(\cC)\to A$ by
\begin{align}
\psi(\af) &=  \psi_D(\af_D)+\psi_Z(\af_Z), \label{psi(af)=psi_D+psi_Z}\\
\0(\af)   &=  \0_D(\af_D)+\0_Z(\af_Z).\label{0(af)=0_D+0_Z}
\end{align}
We shall show that the properties of $\psi$ and $\0$ are determined by the local behavior of these maps. More precisely, given $x,y\in\ob{\cC}$, we denote by $\{x,y\}$ the full subcategory of $C$ whose objects are $x$ and $y$. Identifying $FI(\{x,y\})$ with $e_{\{x,y\}}FI(\cC)e_{\{x,y\}}\sst FI(\cC)$, we have the following result.
\begin{lem}\label{psi-and-0-on-FI(xy)}
	The map $\psi$ defined by \cref{psi(af)=psi_D+psi_Z} is a homomorphism if and only if for all $x<y$ and $\af,\bt\in FI(\{x,y\})$
	\begin{align}
		 \psi_x(\af_{xx}e_{xx})\psi_Z(\bt_{xy}e_{xy})&=\psi_Z(\af_{xx}\bt_{xy}e_{xy}),\label{psi_D(af_xx.e_xx)psi_Z(bt_xy.e_xy)=psi_Z(af_xx.bt_xy.e_xy)}\\
		 \psi_Z(\bt_{xy}e_{xy})\psi_y(\af_{yy}e_{yy})&=\psi_Z(\bt_{xy}\af_{yy}e_{xy}).\label{psi_Z(bt_xy.e_xy)psi_D(af_yy.e_yy)=psi_Z(bt_xy.af_yy.e_xy)}
	\end{align}
	 Similarly, $\0$ given by \cref{0(af)=0_D+0_Z} is an anti-homomorphism if and only if for all $x<y$ and $\af,\bt\in FI(\{x,y\})$
	 \begin{align}
	 \0_Z(\bt_{xy}e_{xy})\0_x(\af_{xx}e_{xx})&=\0_Z(\af_{xx}\bt_{xy}e_{xy}),\label{0_Z(bt_xy.e_xy)0_D(af_xx.e_xx)=0_Z(af_xx.bt_xy.e_xy)}\\
	 \0_y(\af_{yy}e_{yy})\0_Z(\bt_{xy}e_{xy})&=\0_Z(\bt_{xy}\af_{yy}e_{xy}).\label{0_D(af_yy.e_yy)0_Z(bt_xy.e_xy)=0_Z(bt_xy.e_xy.af_xx)}
	 \end{align}
\end{lem}
\begin{proof}
	Since $\psi_Z$ and $\psi_D$ are homomorphisms, it is clear from \cref{psi(af)=psi_D+psi_Z} that $\psi$ is a homomorphism if and only if for all $\af\in D(\cC)$ and $\bt\in FZ(\cC)$
	\begin{align}
		\psi_D(\af)\psi_Z(\bt)&=\psi_Z(\af\bt),\label{psi_D(af)psi_Z(bt)=psi_Z(af.bt)}\\
		\psi_Z(\bt)\psi_D(\af)&=\psi_Z(\bt\af).\label{psi_Z(bt)psi_D(af)=psi_Z(bt.af)}
	\end{align}
	Given arbitrary $x<y$, by \cref{psi(af_xe-e_xy)=vf(e_x)psi(af)vf(e_y),vf(e_y)psi(af)vf(e_x)=vf(e_x)0(af)vf(e_y)=vf(e_x)psi(af)vf(e_x)=vf(e_x)0(af)vf(e_x)=0} we have
	\begin{align*}
		\vf(e_x)\psi_Z(\af\bt)\vf(e_y)+\vf(e_y)\psi_Z(\af\bt)\vf(e_x)=\psi_Z(\af_{xx}\bt_{xy}e_{xy}).
	\end{align*}
	Now, since $e_x$ is a central idempotent of $D(\cC)$, we obtain by \cref{psi_x(af_xx.e_xx)=vf(tl-af_xx.e_xx),psi(af)=vf(tl-af)-and-0(af)=vf(tl-tl-af),psi(af_xe-e_xy)=vf(e_x)psi(af)vf(e_y),vf(e_y)psi(af)vf(e_x)=vf(e_x)0(af)vf(e_y)=vf(e_x)psi(af)vf(e_x)=vf(e_x)0(af)vf(e_x)=0,vf(e)vf(r)=vf(r)vf(e),vf(af_xxe_x)=vf(e_x)vf(af)vf(e_x)}
	\begin{align*}
	\vf(e_x)\psi_D(\af)\psi_Z(\bt)\vf(e_y)&=\vf(e_x)\vf(\tl\af)\psi_Z(\bt)\vf(e_y)\\
	&=\vf(e_x)\vf(\tl\af)\vf(e_x)\psi_Z(\bt_{xy}e_{xy})\\
	&=\vf(\tl\af_{xx}e_{xx})\psi_Z(\bt_{xy}e_{xy})\\
	&=\psi_x(\af_{xx}e_{xx})\psi_Z(\bt_{xy}e_{xy})
	\end{align*}
and
\begin{align*}
	\vf(e_y)\psi_D(\af)\psi_Z(\bt)\vf(e_x)&=\vf(e_y)\vf(\tl\af)\psi_Z(\bt)\vf(e_x)\\
	&=\vf(\tl\af)\vf(e_y)\psi_Z(\bt)\vf(e_x)=0.
	\end{align*}
	Since also
	\begin{align*}
		\vf(e_x)\psi_Z(\af\bt)\vf(e_x)=0=\vf(e_x)\psi_D(\af)\psi_Z(\bt)\vf(e_x)
	\end{align*}
thanks to \cref{psi(af)=vf(tl-af)-and-0(af)=vf(tl-tl-af),vf(e_y)psi(af)vf(e_x)=vf(e_x)0(af)vf(e_y)=vf(e_x)psi(af)vf(e_x)=vf(e_x)0(af)vf(e_x)=0,vf(e)vf(r)=vf(r)vf(e)}, we see that \cref{psi_D(af)psi_Z(bt)=psi_Z(af.bt)} is equivalent to \cref{psi_D(af_xx.e_xx)psi_Z(bt_xy.e_xy)=psi_Z(af_xx.bt_xy.e_xy)} in view of \cref{a=b-in-A}. Similarly, \cref{psi_Z(bt)psi_D(af)=psi_Z(bt.af)} is equivalent to \cref{psi_Z(bt_xy.e_xy)psi_D(af_yy.e_yy)=psi_Z(bt_xy.af_yy.e_xy)}.
	
	The proof of the statement for $\0$ is analogous.
\end{proof}

\begin{lem}\label{vf|_FI(xy)}
	Let $\cC=\cC(P,R)$, where $R$ is a commutative ring and $P$ is a quasiordered set such that $1<|\x|<\infty$ for every class $\x\in \bar{P}$. Let $A$ be an $R$-algebra. Then for each $R$-linear Jordan isomorphism $\vf:FI(\cC)\to A$ and for every $\x\in \bar{P}$ there exists a decomposition of $\vf_{\x}$ into the sum $\psi_{\x}+\0_{\x}$ of a homomorphism $\psi_{\x}$ and an anti-homomorphism $\0_{\x}$, such that  \cref{psi_D(af_xx.e_xx)psi_Z(bt_xy.e_xy)=psi_Z(af_xx.bt_xy.e_xy),psi_Z(bt_xy.e_xy)psi_D(af_yy.e_yy)=psi_Z(bt_xy.af_yy.e_xy),0_Z(bt_xy.e_xy)0_D(af_xx.e_xx)=0_Z(af_xx.bt_xy.e_xy),0_D(af_yy.e_yy)0_Z(bt_xy.e_xy)=0_Z(bt_xy.e_xy.af_xx)} hold.
\end{lem}
\begin{proof}
	Let $\bar x<\bar y$ and consider $Q\sst P$, such that $\bar Q=\{\bar x,\bar y\}$. Observe that $Q$ is a finite quasiordered set, whose classes contain at least 2 elements. Moreover,
	\begin{align*}
		FI(Q,R)\cong FI(\{\bar x,\bar y\}),
	\end{align*}
	so, the restriction $\vf_{\bar x,\bar y}$ of $\vf$ to $FI(\{\bar x,\bar y\})$ can be identified with a Jordan isomorphism $FI(Q,R)\to B$, where $B=\vf(FI(\{\bar x,\bar y\}))$ is an $R$-algebra by \cref{vf(eRe)-subring} and $R$-linearity of $\vf$. It follows from Case 2 of~\cite{Akkurts-Barker} that
	\begin{align}\label{vf_bar_x_bar_y=vf_1+vf_2}
		\vf_{\bar x,\bar y}=\vf_1+\vf_2,
	\end{align}
	where $\vf_1,\vf_2:FI(\{\bar x,\bar y\})\to B$ are a homomorphism and an anti-homomorphism, respectively. Moreover,
	\begin{align}\label{vf_1-and-vf_2-on-FI(xy)}
		\vf_1(\af)=\vf_{\bar x,\bar y}(\af)f,\ \vf_2(\af)=\vf_{\bar x,\bar y}(\af)g
	\end{align}
	for some pair of central orthogonal idempotents $f,g\in B$ whose sum is the identity of $B$.
	
	By the construction of $f$ in~\cite{Akkurts-Barker} we see that $f=f_{\bar x}+f_{\bar y}$, where $f_{\bar x}$ and $f_{\bar y}$ are orthogonal idempotents, $f_{\bar x}$ is a polynomial of the values of $(\vf_{\bar x,\bar y})_{\bar x}=\vf_{\bar x}$, and $f_{\bar y}$ is a polynomial of the values of $\vf_{\bar y}$. The idempotent $g$ also has a similar decomposition $g=g_{\bar x}+g_{\bar y}$. Therefore, for all $\af\in D(\cC)_{\bar x}$ using \cref{vf(e)vf(r)=vf(r)vf(e),vf(e)vf(r)=0,vf_1-and-vf_2-on-FI(xy)} we have
	\begin{align}\label{vf_1-and-vf_2-on-D(C)}
		\vf_1(\af)=\vf_{\bar x,\bar y}(\af)f=\vf_{\bar x}(\af)f=\vf_{\bar x}(\af)\vf(e_{\bar x})(f_{\bar x}+f_{\bar y})=\vf_{\bar x}(\af)f_{\bar x},
	\end{align}
	which shows that $(\vf_1)_{\bar x}$ depends only on $\bar x$ and does not depend on $\bar y$ with $\bar x<\bar y$. By the similar reason $(\vf_2)_{\bar x}$ depends only on $\bar x$. Thus, we may define
	\begin{align}\label{psi_D-and-0_D-from-vf_1-and-vf_2-on-D(C)}
		\psi_{\bar x}=(\vf_1)_{\bar x},\  \ \0_{\bar x}=(\vf_2)_{\bar x}.
	\end{align}
	It follows from \cref{vf_bar_x_bar_y=vf_1+vf_2} that $\vf_{\x}=\psi_{\bar x}+\0_{\bar x}$.
	
	Now take $u\in\bar x$, $v\in\bar y$ and denote by $\e_{uv}\in\mor{\bar x}{\bar y}$, $\e_{uu}\in\mor{\bar x}{\bar x}$, $\e_{vv}\in\mor{\bar y}{\bar y}$ the corresponding matrix units. Analyzing the proof of Case 2 of~\cite{Akkurts-Barker}, one has for all $r\in R$
	\begin{align}\label{vf(r.e_uv.e_bar_x_bary)f=f_uv}
		\vf(r\e_{uv}e_{\bar x\bar y})f=\vf(\e_{uu}e_{\bar x\bar x})\vf(r\e_{uv}e_{\bar x\bar y})\vf(\e_{vv}e_{\bar y\bar y}).
	\end{align}
	But $e_{\bar x}=\e_{uu}e_{\bar x\bar x}+(e_{\bar x}-\e_{uu}e_{\bar x\bar x})$, where
	\begin{align*}
		(e_{\bar x}-\e_{uu}e_{\bar x\bar x})\cdot r\e_{uv}e_{\bar x\bar y}=r\e_{uv}e_{\bar x\bar y}\cdot (e_{\bar x}-\e_{uu}e_{\bar x\bar x})=0.
	\end{align*}
	Hence, $\vf(\e_{uu}e_{\bar x\bar x})\vf(r\e_{uv}e_{\bar x\bar y})=\vf(e_{\bar x})\vf(r\e_{uv}e_{\bar x\bar y})$ by \cref{vf(e)vf(r)=0}. Analogously, we obtain $\vf(r\e_{uv}e_{\bar x\bar y})\vf(\e_{vv}e_{\bar y\bar y})=\vf(r\e_{uv}e_{\bar x\bar y})\vf(e_{\bar y})$. It follows from \cref{vf(r.e_uv.e_bar_x_bary)f=f_uv,vf_1-and-vf_2-on-FI(xy),vf(e_x)psi(af)vf(e_y)=vf(e_x)vf(af)vf(e_y),psi(af_xe-e_xy)=vf(e_x)psi(af)vf(e_y)} that
	\begin{align*}
		\vf_1(r\e_{uv}e_{\bar x\bar y})=\vf(e_{\bar x})\vf(r\e_{uv}e_{\bar x\bar y})\vf(e_{\bar y})=\psi_Z(r\e_{uv}e_{\bar x\bar y}).
	\end{align*}
	Consequently,
	\begin{align}\label{vf_1-coincides-with-psi}
		\vf_1(\af_{\bar x\bar y}e_{\bar x\bar y})=\psi_Z(\af_{\bar x\bar y}e_{\bar x\bar y})
	\end{align}
	for arbitrary $\af_{\bar x\bar y}\in\mor{\bar x}{\bar y}$. Similarly
	\begin{align*}%\label{vf_2-coincides-with-0}
		\vf_2(\af_{\bar x\bar y}e_{\bar x\bar y})=\0_Z(\af_{\bar x\bar y}e_{\bar x\bar y}).
	\end{align*}
	Since $\vf_1$ is a homomorphism, we have by \cref{psi_D-and-0_D-from-vf_1-and-vf_2-on-D(C),vf_1-coincides-with-psi}
	\begin{align*}
		\psi_{\bar x}(\af_{\bar x\bar x}e_{\bar x\bar x})\psi_Z(\bt_{\bar x\bar y}e_{\bar x\bar y})&=\vf_1(\af_{\bar x\bar x}e_{\bar x\bar x})\vf_1(\bt_{\bar x\bar y}e_{\bar x\bar y})=\vf_1(\af_{\bar x\bar x}\bt_{\bar x\bar y}e_{\bar x\bar y})\\
		&=\psi_Z(\af_{\bar x\bar x}\bt_{\bar x\bar y}e_{\bar x\bar y}).
	\end{align*}
	Thus, $\psi_{\bar x}$ satisfies \cref{psi_D(af_xx.e_xx)psi_Z(bt_xy.e_xy)=psi_Z(af_xx.bt_xy.e_xy)}. The proof that it also satisfies \cref{psi_Z(bt_xy.e_xy)psi_D(af_yy.e_yy)=psi_Z(bt_xy.af_yy.e_xy)} is similar. Analogously one proves \cref{0_Z(bt_xy.e_xy)0_D(af_xx.e_xx)=0_Z(af_xx.bt_xy.e_xy),0_D(af_yy.e_yy)0_Z(bt_xy.e_xy)=0_Z(bt_xy.e_xy.af_xx)}.
\end{proof}

We are ready to prove the main result of \cref{jord-iso-FI(P_R)}.

\begin{thrm}\label{vf=psi+0}
Let $R$ be a commutative ring and $P$ a quasiordered set such that $1< |\x| < \infty$ for all $\x\in \bar{P}$. Let $A$ be an $R$-algebra. Then each $R$-linear Jordan isomorphism $\vf:FI(P,R)\to A$ is the sum of a homomorphism and an anti-homomorphism.
\end{thrm}
\begin{proof}
 By	\cref{vf|_FI(xy)}, for every $\bar x \in \bar P$ there is a decomposition $\vf_{\bar x}=\psi_{\bar x}+\0_{\bar x}$ of $\vf_{\bar x}$ where $\psi_{\bar x}$ and $\0_{\bar x}$ are a homomorphism and an anti-homomorphism, respectively, for which \cref{psi_D(af_xx.e_xx)psi_Z(bt_xy.e_xy)=psi_Z(af_xx.bt_xy.e_xy),psi_Z(bt_xy.e_xy)psi_D(af_yy.e_yy)=psi_Z(bt_xy.af_yy.e_xy),0_Z(bt_xy.e_xy)0_D(af_xx.e_xx)=0_Z(af_xx.bt_xy.e_xy),0_D(af_yy.e_yy)0_Z(bt_xy.e_xy)=0_Z(bt_xy.e_xy.af_xx)} hold. By \cref{vf|_D(C)-is-the-sum-of-psi-and-theta,psi-and-0-on-FI(xy)} this leads to a decomposition $\vf|_{D(\cC)}=\psi_D+\0_D$, such that the maps $\psi$ and $\0$ given by \cref{psi(af)=psi_D+psi_Z,0(af)=0_D+0_Z} are a homomorphism and an anti-homomorphism, respectively. Obviously, $\vf=\psi+\0$.
\end{proof}

\bibliography{bibl}{}
\bibliographystyle{acm}

%\bibliography{bibl}{}
%\bibliographystyle{acm}

\end{document}